\documentclass[10pt,reqno]{amsart}

\usepackage[dvipsnames]{xcolor}
\usepackage{amsmath,amssymb,amsthm,amsfonts,enumerate,tikz,bm}
\usepackage{mathrsfs}
\usetikzlibrary{matrix,arrows,positioning,calc}
\usepackage[all]{xy}
\usepackage[bookmarksnumbered,colorlinks]{hyperref}
\usepackage{url}
\numberwithin{equation}{section}
\usepackage{graphicx}
\usepackage[T1]{fontenc}
\usepackage{textcomp}
\usepackage{palatino,helvet}

\usepackage{footmisc}

\newtheorem{definition}{Definition}[section]
\newtheorem{remark}[definition]{Remark}
\newtheorem{example}[definition]{Example}
\newtheorem{theorem}[definition]{Theorem}
\newtheorem{proposition}[definition]{Proposition}
\newtheorem{lemma}[definition]{Lemma}
\newtheorem{corollary}[definition]{Corollary}
\theoremstyle{remark}

\usepackage{paralist}

\usepackage{scrextend}
\deffootnote{2em}{0em}{\thefootnotemark\quad}

\newcommand{\mcA}{\mathcal{A}}
\newcommand{\mcB}{\mathcal{B}}

\newcommand{\mcI}{\mathcal{I}}

\newcommand{\mcP}{\mathcal{P}}

\newcommand{\mcX}{\mathcal{X}}

\newcommand{\mcY}{\mathcal{Y}}
\newcommand{\mcZ}{\mathcal{Z}}
\newcommand{\mcGP}{\mathcal{GP}}

\newcommand{\mcGI}{\mathcal{GI}}

\newcommand{\mfC}{\mathfrak{C}}
\newcommand{\mfF}{\mathfrak{F}}

\newcommand{\mfGF}{\mathfrak{GF}}
\newcommand{\mfGC}{\mathfrak{GC}}

\newcommand{\Mod}{\mathsf{Mod}}

\newcommand{\Ch}{\mathsf{Ch}}
\newcommand{\Qcoh}{\mathfrak{Qcoh}}

\newcommand{\LPRes}{\mathsf{p.l.res}}
\newcommand{\RPRes}{\mathsf{p.r.res}}

\newcommand{\resdim}{{\rm l.resdim}}
\newcommand{\coresdim}{{\rm r.resdim}}
\newcommand{\pd}{{\rm pd}}
\newcommand{\id}{{\rm id}}

\newcommand{\Hom}{{\rm Hom}}
\newcommand{\Ext}{{\rm Ext}}
\newcommand{\Hin}{\mathcal{H}om}
\newcommand{\Exin}{\mathcal{E}xt}

\newcommand{\Ker}{{\rm Ker}}
\newcommand{\Coker}{{\rm CoKer}}

\makeatletter
\def\@seccntformat#1{%
  \protect\textup{\protect\@secnumfont
    \ifnum\pdfstrcmp{section}{#1}=0 \scshape\bfseries\fi% section # in \scshape and \bfseries
    \ifnum\pdfstrcmp{subsection}{#1}=0 \bfseries\fi% subsection # in \bfseries
    \csname the#1\endcsname
    \protect\@secnumpunct
  }%
}
\makeatother

\begin{document}

\title{Balanced systems for Hom}

\author{V\'ictor Becerril}
\address[V. Becerril]{Centro de Ciencias Matem\'aticas. Universidad Nacional Aut\'onoma de M\'exico. 
 58089. Morelia, Michoac\'an, M\'EXICO}
\email{victorbecerril@matmor.unam.mx}

\author{Octavio Mendoza}
\address[O. Mendoza]{Instituto de Matem\'aticas. Universidad Nacional Aut\'onoma de M\'exico. Circuito Exterior, Ciudad Universitaria. CP04510. Mexico City, MEXICO}
\email{omendoza@matem.unam.mx}

\author{Marco A. P\'erez}
\address[M. A. P\'erez]{Instituto de Matem\'atica y Estad\'istica ``Prof. Ing. Rafael Laguardia''. Facultad de Ingenier\'ia. Universidad de la Rep\'ublica. CP11300. Montevideo, URUGUAY}
\email{mperez@fing.edu.uy}

\maketitle

\begin{abstract}
From the notion of generator and its dual in Auslander-Buchweitz approximation theory, we present the concept of finite balanced system as a tool to induce balanced pairs for Hom bifunctors with domain determined by classes of objects with finite left and right resolution dimension relative to these pairs.

This approach to balance of Hom will cover several well known ambients where its right derived functors are obtained relative to certain classes of objects in abelian categories, such as Gorenstein projective and injective modules and chain complexes, Gorenstein modules relative to Auslander and Bass classes, among others. 
\end{abstract}

%\setcounter{tocdepth}{2}
%\tableofcontents

\pagestyle{myheadings}
\markboth{\rightline {\scriptsize V. Becerril, O. Mendoza and M. A. P\'{e}rez}}
         {\leftline{\scriptsize Balanced systems for Hom}}

%%%%%%%%%%%%%%%%%%%%%%%%%%%%%%%%%%%%%
%%%%%%%%%%%%%%%%%%%%%%%%%%%%%%%%%%%%%
%%%%%%%%%%%%%%%%%%%%%%%%%%%%%%%%%%%%%
%%%%%%%%%%%%%%%%%%%%%%%%%%%%%%%%%%%%%

\section*{\textbf{Introduction}}

Given an associative ring $R$ with identity, and two left $R$-modules $M$ and $N$, a well known result from basic homological algebra states that the abelian group of $i$-fold extensions $\Ext^i_R(M,N)$, with $i > 0$ a positive integer, can be computed from an exact left projective resolution of $M$ or from an exact right injective resolution of $N$. This is a consequence of the fact that the bifunctor $\Hom_R(-,\sim)$ is (right) balanced on the product category $\Mod(R) \times \Mod(R)$ by the subcategory $\mathcal{P}(R) \times \mathcal{I}(R)$, where $\Mod(R)$ denotes the category of left $R$-modules, and $\mathcal{P}(R)$ and $\mathcal{I}(R)$ the classes of projective and injective left $R$-modules, respectively. Using the terminology of Chen \cite[Def. 1.1]{Chen}, one can summarize the previous by saying that $(\mathcal{P}(R),\mathcal{I}(R))$ is a \emph{balanced pair}. Specifically, two classes $\mcX$ and $\mcY$ of left $R$-modules form a balanced pair if:
\begin{enumerate}
\item $\mcX$ is precovering and $\mcY$ is preenveloping.

\item Every left $R$-module has a left resolution by objects in $\mcX$ that is exact after applying $\Hom_R(-,Y)$ for every $Y \in \mcY$.

\item Every left $R$-module has a right resolution by objects in $\mcY$ that is exact after applying $\Hom_R(X,-)$ for every $X \in \mcX$.
\end{enumerate}

This notion of balance has turned out to be key in the theory of derived functors, with connections to relative homological algebra. For example, it has been proved by Enochs and Jenda's \cite[Thm. 12.1.4]{EnJen00} that if $R$ is an Iwanaga-Gorenstein ring, then $\Hom_{R}(-,\sim)$ is balanced on $\Mod(R) \times \Mod(R)$ by $\mcGP(R) \times \mcGI(R)$, where $\mcGP(R)$ and $\mcGI(R)$ denote the classes of Gorenstein projective and Gorenstein injective left $R$-modules. In other words, $(\mathcal{GP}(R),\mathcal{GI}(R))$ is a balanced pair. This result was later generalized by Holm \cite[Thm. 3.6]{Holm05} for arbitrary rings, in the sense that $\Hom_R(-,\sim)$ is balanced by $\mcGP(R) \times \mcGI(R)$ on $\mcGP(R)^\wedge \times \mcGI(R)^{\vee}$, the product subcategory of $\Mod(R) \times \Mod(R)$ formed by the pairs $(M,N)$ where $M$ has finite Gorenstein projective dimension and $N$ has finite Gorenstein injective dimension. More specifically, conditions (1), (2) and (3) above can be restated in the following way:
\begin{enumerate}[(a)]
\item Every left $R$-module in $\mathcal{GP}(R)^\wedge$ (that is, with finite Gorenstein projective dimension) has a Gorenstein projective precover, and every left $R$-module in $\mathcal{GI}(R)^\vee$ has a Gorenstein injective preenvelope. More specifically, one can find for objects in $\mathcal{GP}(R)^\wedge$ Gorenstein projective epic precovers with kernel in $\mathcal{P}(R)^\wedge$, and for objects in $\mathcal{GI}(R)^\vee$ Gorenstein injective monic preenvelopes with cokernel in $\mathcal{I}(R)^\vee$.

\item Every left $R$-module in $\mathcal{GP}(R)^\wedge$ has an exact left Gorenstein projective resolution which remains exact after applying $\Hom_R(-,N)$ with $N \in \mathcal{GI}(R)$.

\item Every left $R$-module in $\mathcal{GI}(R)^\vee$ has an exact right Gorenstein injective resolution which remains exact after applying $\Hom_R(M,-)$ with $M \in \mathcal{GP}(R)$.
\end{enumerate}
So in some sense, the concept of balanced pair becomes \emph{relative} to the subcategory $\mathcal{GP}(R)^\wedge \times \mathcal{GI}(R)^\vee$. One important aspect to point out is that conditions (b) and (c) are based on Enochs and Jenda's \cite[Lem. 1.3]{EnochsJenda95}, which asserts that $\Ext^i_R(M,N) = 0$ for every $M \in \mathcal{P}(R)^\wedge$ and $N \in \mathcal{GI}(R)$ (or for every $M \in \mathcal{GP}(R)$ and $N \in \mathcal{I}(R)^\vee$). 

Motivated by the previous example, we aim to present a relativization of the notion of balanced pair. This relativization, called \emph{finite balanced system}, will be formed by a quadruple $[(\mcX,\omega);(\nu,\mcY)]$ of classes of objects $\mcX$, $\mcY$, $\omega$ and $\nu$ in an abelian category $\mcA$ which somehow mimic conditions (a), (b) and (c) above, along with \cite[Lem. 1.3]{EnochsJenda95}. Finite balanced systems will cover situations of balance of $\Hom_R(-,\sim)$ that cannot be approached using only the theory of balanced pairs. More specifically, the interplay between the classes in a finite balanced system $[(\mcX,\omega);(\nu,\mcY)]$ will comprise sufficient conditions so that $(\mcX,\mcY)$ is a balanced pair with respect to $(\mcX^\wedge,\mcY^\vee)$ (as in the mentioned example formed by Gorenstein projective and Gorenstein injective modules over an arbitrary ring). On the other hand, $\omega$ and $\nu$ will act as auxiliary classes for $\mcX$ and $\mcY$, in the sense that every object $X \in \mcX$ can be embedded into an object $W \in \omega$ such that $W / X \in \mcX$, and that for every object $Y \in \mcY$ there is an epimorphism $V \twoheadrightarrow Y$ with kernel in $\mcY$ (keep in mind the definitions of Gorenstein projective and Gorenstein injective left $R$-modules). In other words, $\mcY$ (resp., $\mcX$) is equipped with a (co)generating class $\nu$ (resp., $\omega$). These, along with some Ext-orthogonality conditions, namely
\[
\Ext^1_{\mcA}(\mcX^\wedge,\nu) = \Ext^1_{\mcA}(\mcX,\nu^\vee) = \Ext^1_{\mcA}(\omega,\mcY^\vee) = \Ext^1_{\mcA}(\omega^\wedge,\mcY) = 0
\] 
(which mimic \cite[Lem. 1.3]{EnochsJenda95}), will allow us to construct relative derived groups $\Ext^i_{(\mcX,\mcY)}(M,N)$ for every pair $(M,N) \in \mcX^\wedge \times \mcY^\vee$, by either taking:
\begin{itemize}
\item a left resolution of $M$ by objects in $\mcX$ which is exact after applying the (contravariant) functor  $\Hom_{\mcA}(-,Y)$ for every $Y \in \mcY$; or 
\item a right resolution of $N$ by objects in $\mcY$ which is exact after applying the (covariant) functor $\Hom_{\mcA}(X,-)$ for every $X \in \mcX$.
\end{itemize} 
In our setting, the constructions of such left and right resolutions will come from Auslander-Buchweitz approximation theory, where the notions of relative injective cogenerator and relative projective generator are essential. 

This paper is organized as follows. In Section \ref{sec:prelims} we recall some preliminary notions from relative homological algebra, such as relative homological dimensions, left and right approximations, cotorsion pairs and relative (co)generators. Section \ref{sec:rel_Ext_functors} is devoted to recall the concepts of proper resolutions and relative derived functors of $\Hom_{\mcA}(-,\sim)$, along with some of their properties. The main results of this paper are presented in Section \ref{sec:main}. First, we present in Definition \ref{def:balanced_pair} the concept of balanced pairs $(\mcX,\mcY)$ relative to a pair $(\tilde{\mcX},\tilde{\mcY})$ of classes of objects in $\mcA$. This is a relativization of the original concept proposed by Chen in \cite{Chen}. These pairs are sources for relative derived functors of $\Hom_{\mcA}(-,\sim)$ with domain in $\tilde{\mcX} \times \tilde{\mcY}$ (see Proposition \ref{CriterioBalance}). Later, in Definition \ref{balanced-couple} we define balanced systems $[(\mcX,\omega); (\nu,\mcY)]$ of pairs of classes of objects in $\mcA$ with respect to a couple $[(\mcX',\mcX''); (\mcY'',\mcY')]$. These collect sufficient conditions for $(\mcX,\mcY)$ to be an admissible balanced pair relative to $(\overline{\mathsf{L}}_\mcX(\mcX',\mcX''),\overline{\mathsf{R}}^{\mcY}(\mcY'',\mcY'))$ (see Proposition \ref{tripleKer-ex}), where $\overline{\mathsf{L}}_\mcX(\mcX',\mcX'')$ is formed by the objects in $\mcX$ or those $M \in \mcX''$ admitting a proper left $\mcX$-resolution with cycles in $\mcX'$ (and $\overline{\mathsf{R}}^{\mcY}(\mcY'',\mcY')$ is defined dually). We also provide examples of balanced systems in relative Gorenstein homological algebra, and provide equivalences that characterize the balance of Hom by relative Gorenstein modules via the existence of certain balanced systems. Later on we shall focus on a particular type of balanced systems called finite and strong finite (Definition \ref{finite-balanced-system}). We give some basic examples and characterizations for these concepts, which will help us to construct more elaborated examples and applications in Section \ref{sec:applications}. We show how to induce, from a strong finite balanced system in the ground category $\mcA$, a balanced system of the same sort in the category of chain complexes of objects in $\mcA$. This method will bring to the category of chain complexes over a ring the examples of balance of Hom by the classes of Gorenstein (Ding) projective and injective modules. We also recover a result of Sather-Wagstaff, Sharif and White \cite{SeanWhite10} concerning balance of Hom by the classes of $C$-Gorenstein projective and $\Hom(C,D)$-Gorenstein injective modules over a commutative Cohen-Macaulay ring with a dualizing module $D$ and a semidualizing module $C$. Other situations of balance regarding Gorenstein flat quasi-coherent sheaves and virtually Gorenstein rings are analyzed as well.

%%%%%%%%%%%%%%%%%%%%%%%%%%%%%%%%%%%%%%%%%%%%%%%%%
%%%%%%%%%%%%%%%%%%%%%%%%%%%%%%%%%%%%%%%%%%%%%%%%%
%%%%%%%%%%%%%%%%%%%%%%%%%%%%%%%%%%%%%%%%%%%%%%%%%
%%%%%%%%%%%%%%%%%%%%%%%%%%%%%%%%%%%%%%%%%%%%%%%%%

\section{Preliminaries}\label{sec:prelims}

Throughout, $\mcA$ always denotes an abelian category (not necessarily with enough projective or injective objects). Monomorphisms and epimorphisms in $\mcA$ may be denoted by $\rightarrowtail$ and $\twoheadrightarrow$, respectively. Classes of objects in $\mcA$ are always assumed to be closed under isomorphisms. The set of natural numbers, denoted by $\mathbb{N}$, will be used as a set of indexes in several constructions. In this paper, $0 \in \mathbb{N}$, and by $\mathbb{N}^\ast$ we shall mean the set of all natural numbers excluding $0$.

Given two objects $M, N \in \mcA$ and $n \in \mathbb{N}$, the Yoneda abelian group of $n$-fold extensions of $N$ by $M$ is denoted by $\Ext^n_{\mcA}(M,N)$. For $n = 0$, $\Ext^0_{\mcA}(M,N)$ is the abelian group $\Hom_{\mcA}(M,N)$ of morphisms from $M$ to $N$. Given $\mcX, \mcY \subseteq \mcA$ two classes of objects of $\mathcal{A}$, the notation $\Ext^n_{\mcA}(\mcX,N) = 0$ means that $\Ext^n_{\mcA}(X,N) = 0$ for every $X \in \mcX$. The equality $\Ext^n_{\mcA}(M,\mcY) = 0$ has a similar meaning. Finally, $\Ext^n_{\mcA}(\mcX,\mcY) = 0$ means that $\Ext^n_{\mcA}(X,Y) = 0$ for every $X \in \mcX$ and $Y \in \mcY$. In the case where $\Ext^i_{\mcA}(\mcX,\mcY) = 0$ for every $i > n$, we write $\Ext^{\geq n+1}_{\mcA}(\mcX,\mcY) = 0$.

%%%%%%%%%%%%%%%%%%%%%%%%%%%%%%%%%%%%%%%%%%%%%%%%%
%%%%%%%%%%%%%%%%%%%%%%%%%%%%%%%%%%%%%%%%%%%%%%%%%

\subsection*{Relative projective and injective dimensions}

Let $\mcX \subseteq \mcA$ and $M \in \mcA$. The \emph{projective dimension of $M$ relative to $\mcX$} is defined as
\[
\pd_{\mcX}(M) := \inf \{ n \in \mathbb{N} \mbox{{\rm \ : \ }} \Ext_{\mcA}^{\geq n+1}(M,\mcX) = 0 \}.
\]
In the case where $\{ n \in \mathbb{N} \mbox{{\rm \ : \ }} \Ext_{\mcA}^{\geq n+1}(M,\mcX) = 0 \} = \emptyset$, we set $\pd_{\mcX}(M) = \infty$. Dually, we have the \emph{injective dimension of $M$ relative to $\mcX$}, denote by $\id_{\mcX}(M)$. Furthermore, for any $\mcY \subseteq \mcA$, we set 
\[
\pd_{\mcX}(\mcY) := \mathrm{sup}\,\{\pd_{\mcX}(Y) \text{ : } Y \in \mcY \},
\] 
and $\id_{\mcX}(\mcY)$ is defined similarly. Moreover, these two dimensions are related by the equality $\pd_{\mcX}(\mcY) = \id_{\mcY}(\mcX)$. If $\mcX = \mcA$, we simply write $\pd(M)$ and $\pd(\mcY)$ for the (absolute) projective dimensions of $M$ and $\mcY$, and similarly for $\id(M)$ and $\id(\mcX)$ if $\mcY = \mcA$. We denote by 
\[
\mcP(\mcA) := \{ M \in \mathcal{A} \mbox{{\rm \ : \ }} \pd(M) = 0 \} \text{ \ and \ } \mcI(\mcA) := \{ M \in \mathcal{A} \mbox{{\rm \ : \ }} \id(M) = 0 \}
\]
the classes of projective and injective objects of $\mcA$.

%%%%%%%%%%%%%%%%%%%%%%%%%%%%%%%%%%%%%%%%%%%%%%%%%
%%%%%%%%%%%%%%%%%%%%%%%%%%%%%%%%%%%%%%%%%%%%%%%%%

\subsection*{Orthogonal complements}

For each $i \in \mathbb{N}^\ast$ we consider the right  orthogonal classes
\[
\mcX^{\perp_i}:=\{M \in \mcA \text{ : } \Ext^i_{\mcA}(\mcX,M) = 0 \} \quad {\rm and} \quad  \mcX^\perp := \bigcap_{i\in \mathbb{N}^\ast}\,\mcX^{\perp_i}.
\] 
Dually, we have the left orthogonal classes ${}^{\perp_i}\mcX$ and ${}^{\perp}\mcX$.

%%%%%%%%%%%%%%%%%%%%%%%%%%%%%%%%%%%%%%%%%%%%%%%%%
%%%%%%%%%%%%%%%%%%%%%%%%%%%%%%%%%%%%%%%%%%%%%%%%%

\subsection*{Relative resolutions} 

Given a class $\mcX \subseteq \mcA$ and an object $M \in \mcA$, a \emph{left $\mcX$-resolution of $M$} is a (not necessarily exact) complex 
\[
\cdots \to X_1 \to X_0 \xrightarrow{\varepsilon} M \to 0
\] 
where $X_k \in \mcX$ for every $k \in \mathbb{N}$. The truncated complex $\cdots \to X_1 \to X_0$ will be denoted by $\mcX_\bullet(M)$. The morphism $\varepsilon$ is called the \emph{augmentation map} of the resolution. Thus, we may sometimes denote left $\mcX$-resolutions as $\varepsilon \colon \mcX_\bullet(M) \to M$.

The \emph{left $\mcX$-resolution dimension of $M$}, denoted $\resdim_{\mcX}(M)$, is defined as the infimum of the set
\[
\{ n \in \mathbb{N} \mbox{{\rm \ : \ }} M \text{ admits an exact left $\mcX$-resolution with $X_k = 0$ for every $k > n$} \}.
\]
Again, if the previous set is empty, we set $\resdim_{\mcX}(M) = \infty$. \emph{Right $\mcX$-resolutions} and \emph{right $\mcX$-resolution dimensions} of objects, denoted $\coresdim_{\mcX}(M)$, are defined dually. Given $\mcY \subseteq \mcA$, we set 
\[
\resdim_{\mcX}(\mcY) := \mathrm{sup}\{ \resdim_{\mcX}(Y) \text{ : } Y \in \mcY \},
\] 
and $\coresdim_{\mcX}(\mcY)$ is defined similarly.

We shall often consider the following classes of objects in $\mcA$:
\begin{align*}
\mcX^\wedge_n & := \{ M \in \mcA \mbox{{\rm \ : \ }} \resdim_{\mcX}(M) \leq n \}, & \mcX^\wedge & := \bigcup_{n \in \mathbb{N}} \mcX^\wedge_n, \\
\mcX^\vee_n & := \{ M \in \mcA \mbox{{\rm \ : \ }} \coresdim_{\mcX}(M) \leq n \}, & \mcX^\vee & := \bigcup_{n \in \mathbb{N}} \mcX^\vee_n.
\end{align*}
In some references, $\resdim_{\mcX}(-)$ and $\coresdim_{\mcX}(-)$ are also known as resolution and coresolution dimensions relative to $\mcX$.

%%%%%%%%%%%%%%%%%%%%%%%%%%%%%%%%%%%%%%%%%%%%%%%%%
%%%%%%%%%%%%%%%%%%%%%%%%%%%%%%%%%%%%%%%%%%%%%%%%%

\subsection*{Approximations} 

Given a class $\mcX \subseteq \mcA$, a morphism $f \colon X \rightarrow M$ in $\mcA$ with $X \in \mcX$ is an \emph{$\mcX$-precover of $M$} if $\Hom_{\mcA}(X', f) \colon \Hom_{\mcA}(X',X) \to \Hom_{\mcA}(X',M)$ is surjective for every $X' \in \mcX$. Moreover, $f$ is said to be \emph{special} if $\Coker(f) = 0$ and $\Ker(f) \in \mcX^{\perp_1}$. We shall freely make use of the dual notions of \emph{$\mcX$-preenvelope} and \emph{special $\mcX$-preenvelope}. The class $\mcX$ is \emph{precovering} if every object of $\mcA$ has a $\mcX$-precover. \emph{Special precovering}, \emph{preenveloping} and \emph{special preenveloping} classes are defined similarly.

%%%%%%%%%%%%%%%%%%%%%%%%%%%%%%%%%%%%%%%%%%%%%%%%%
%%%%%%%%%%%%%%%%%%%%%%%%%%%%%%%%%%%%%%%%%%%%%%%%%

\subsection*{Relative (co)generators} 

Let $(\mcX,\omega)$ be a pair of classes of objects in $\mcA$. We recall that $\omega$ is a \emph{relative quasi-cogenerator} in $\mcX$ if for every $X \in \mcX$ there is a short exact sequence $X \rightarrowtail W \twoheadrightarrow X'$ such that $W \in \omega$ and $X' \in \mcX$. If in addition $\omega \subseteq \mcX$, then $\omega$ is said to be a \emph{relative cogenerator} in $\mcX$. We also recall from Becerril, Mendoza and Santiago's \cite{BMS} that $\omega$ is a \emph{generator} in $\mcX$ if for any $X\in\mcX$ there is an epimorphism $W\twoheadrightarrow X$ with $W\in\omega$.\footnote{Not to be confused with the usual terminology of generator in category theory, where one asks that $\omega$ is a set and $X$ is the epimorphic image of a coproduct of objects in $\omega$.} Dually, we have the notions of \emph{relative} (\emph{quasi}) \emph{generators}. For example, $\mcI(\mcA)$ (resp., $\mcP(\mcA)$) is a relative cogenerator (resp., generator) in $\mcA$ if, and only if, $\mcA$ has enough injective (resp., projective) objects. Other well known examples can be found in the category $\Mod(R)$ of left $R$-modules. In particular, the class $\mathcal{P}(R)$ (resp., $\mathcal{I}(R)$) of projective (resp., injective) left $R$-modules is both relative generating and relative cogenerating in the class $\mathcal{GP}(R)$ (resp., $\mathcal{GI}(R)$) of Gorenstein projective (resp., Gorenstein injective) left $R$-modules.

%%%%%%%%%%%%%%%%%%%%%%%%%%%%%%%%%%%%%
%%%%%%%%%%%%%%%%%%%%%%%%%%%%%%%%%%%%%
%%%%%%%%%%%%%%%%%%%%%%%%%%%%%%%%%%%%%
%%%%%%%%%%%%%%%%%%%%%%%%%%%%%%%%%%%%%

\section{Relative extension bifunctors}\label{sec:rel_Ext_functors}

In this section we recall how to define relative right derived functors from Hom. Let $\mcX \subseteq \mcA$ be a class of objects of $\mcA$. We recall that a chain complex
\[
A_\bullet = \cdots \to A_1 \to A_0 \to A_{-1} \to \cdots 
\]
in $\mcA$ is \emph{$\Hom_{\mcA}(\mcX,-)$-acyclic} if the induced complex of abelian groups
\[
\Hom_{\mcA}(X,A_\bullet) = \cdots \to \Hom_{\mcA}(X,A_1) \to \Hom_{\mcA}(X,A_0) \to \Hom_{\mcA}(X,A_{-1}) \to \cdots
\]
is exact for every $X \in \mcX$. Similarly, we have the concept of \emph{$\Hom_{\mcA}(-,\mcX)$-acyclic} complexes. This type of acyclicity can be characterized for bounded complexes as shown in Lemma \ref{ex-xyres} below. Given a bounded below complex 
\[
M_{\bullet \geq -1} \equiv \cdots \to M_2 \xrightarrow{d_2} M_1 \xrightarrow{d_1} M_0 \xrightarrow{d_0} M_{-1} \to 0
\] 
of objects in $\mcA$, for each $i \in \mathbb{N}$ let $\mu_i \colon Z_i(M_\bullet^+) \rightarrowtail M_i$ denote the kernel of the differential $d_i \colon M_i \to M_{i-1}$. Since $d_{i-1} \circ d_i = 0$, there exists a unique morphism $d'_i \colon M_i \to Z_{i-1}(M_{\bullet \geq -1})$ such that $d_i = \mu_{i-1} \circ d'_i$. By taking $d'_0 := d_0$ and $Z_{-1}(M_{\bullet \geq -1}) = M_{-1}$, we have that the complex $M_{\bullet \geq -1}$ produces for every $i \in \mathbb{N}$ the short complex 
\[
\eta_i \equiv Z_i(M_{\bullet \geq -1}) \overset{\mu_i}\rightarrowtail M_i \xrightarrow{d'_i} Z_{i-1}(M_{\bullet \geq -1}).
\]

The statement of the following result comes from the arguments appearing in \cite[Proof of Lem. 2.4]{Chen}, and its proof follows by standard arguments.

\begin{lemma}[acyclicity condition for Hom]\label{ex-xyres} 
The complex $M_{\bullet \geq -1}$ is $\Hom_{\mcA}(\mcX,-)$-acyclic if, and only if, $\eta_i$ is $\Hom_{\mcA}(\mcX,-)$-acyclic for every $i \in \mathbb{N}$. Moreover, if any of these conditions holds true and for every $j \in \mathbb{N} \cup \{ -1\}$ there is an epimorphism $X_j \to Z_j(M_{\bullet \geq -1})$ with $X_j \in \mcX$, then $M_{\bullet \geq -1}$ is exact. 
\end{lemma}

\begin{definition}
A left (resp., right) $\mcX$-resolution is called \textbf{proper} if it is $\Hom_{\mcA}(\mcX,-)$-acyclic (resp., $\Hom_{\mcA}(-,\mcX)$-acyclic). The classes of objects in $\mcA$ admitting a proper left and a proper right $\mcX$-resolution will be denoted by $\LPRes(\mcX)$ and $\RPRes(\mcX)$, respectively.
\end{definition}

\begin{remark}
One can use the previous lemma to provide an alternative description of the class $\LPRes(\mcX)$ in the case where $\mcX$ is a generator in $\mcA.$ This description is stated below in terms of relative Gorenstein objects in the sense of \cite[Def. 3.1]{BMS}. Specifically, $M \in \mcA$ is \textbf{weak $\bm{\mcX}$-Gorenstein injective} if $M \in \mcX^\perp$ and if there exists an exact left $\mcX$-resolution $\varepsilon \colon \mcX_\bullet(M) \to M$ with ${\rm Ker}(\varepsilon) \in \mcX^\perp$ and $Z_i(\mcX_\bullet(M)) \in \mcX^\perp$ for every $i \in \mathbb{N}$. So if $\mathcal{WGI}_{\mcX}(\mcA)$ denotes the class of weak $\mcX$-Gorenstein injective objects in $\mcA$, then one has that 
\[
\LPRes(\mcX) = \mathcal{WGI}_{\mcX}(\mcA).
\]
provided that $\mcX$ is a generator in $\mcA$ and $\pd_{\mcX}(\mcX) = 0$. One example of this situation is obtained by setting $\mcX = \omega$ satisfying $\Ext^{\geq 1}_{\mcA}(\omega,\omega) = 0$, where $\mathcal{WGI}_{\omega}(\mcA)$ coincides with the class of dual Cohen-Macaulay objects relative to $\omega$ (see Beligiannis and Reiten's \cite[pp. 95]{BeligiannisReiten}). 
\end{remark}

The following result is basically the one that appears in \cite[Lem. 8.2.1]{EnJen00}, but we have removed the condition that the given class $\mcX \subseteq \mcA$ is precovering. Its proof is similar to the mentioned reference and uses Lemma \ref{ex-xyres}.

\begin{lemma}[relative horseshoe lemma]\label{HorseshoeL}
Let $\mcX \subseteq \mcA$ be a class of objects closed under finite coproducts, and 
\[
\mathbb{M} \equiv 0 \to M ' \xrightarrow{\alpha} M \xrightarrow{\beta} M'' \to 0
\] 
be a $\Hom_{\mcA}(\mcX,-)$-acyclic complex. If $\varepsilon' \colon \mcX_\bullet(M') \to M'$ and $\varepsilon'' \colon \mcX_\bullet(M'') \to M''$ are proper left $\mcX$-resolutions, then there is a proper left $\mcX$-resolution $\varepsilon \colon \mcX_\bullet(M) \to M$ and a degreewise split exact sequence of complexes 
\[
0 \to \mcX_\bullet(M') \xrightarrow{\mcX_\bullet(\alpha)} \mcX_\bullet(M)\xrightarrow{\mcX_\bullet(\beta)} \mcX_\bullet(M'') \to 0
\]
such that $\varepsilon \circ [\mcX_\bullet(\alpha)]_0 = \alpha \circ \varepsilon'$ and $\varepsilon'' \circ [\mcX_\bullet(\beta)]_0 = \beta \circ \varepsilon$, where $\mcX_\bullet(\alpha)$ and $\mcX_\bullet(\beta)$ are the chain maps induced from $\alpha$ and $\beta$ in the usual way.
\end{lemma}

In Holm's \cite[Lem. 1.7]{Holm} a similar result is stated but with the assumption that the sequence $\mathbb{M}$ is exact. After a careful revision of this reference, and using Lemma \ref{ex-xyres}, one can note that exactness is not needed. Another result form \cite{Holm} concerning proper left resolutions is the following comparison lemma in the category of left $R$-modules, but its proof carries over to any abelian category.

\begin{lemma}[uniqueness up to homotopy]\label{ComparisonT}
Let $\mcX \subseteq \mcA$ be a class of objects, and $\varepsilon \colon \mcX_\bullet(M) \to M$ and $\phi \colon \mcX_\bullet(N) \to N$ be left $\mcX$-resolutions of $M$ and $N$, where $\phi$ is proper. Then, for any morphism $f \colon M \to N$ in $\mcA$, there is a chain map 
\[
\mcX_\bullet(f) \colon \mcX_\bullet(M) \to \mcX_\bullet(N)
\]
such that $f \circ \varepsilon = \phi \circ [\mcX_\bullet(f)]_0$. Moreover, $\mcX_\bullet(f)$ is unique up to chain homotopy.
\end{lemma}

The previous statement is slightly more general than the one appearing in \cite{Holm}, where the resolutions are assumed to be exact. Actually, for the construction of $\mcX_\bullet(f)$ and chain homotopies, exactness is not needed but $\Hom_{\mcA}(\mcX,-)$-acyclicity. 

Having recalled all the previous properties for proper resolutions, one can define left and right derived functors of $\Hom_{\mcA}(-,\sim) \colon \mcA^{\rm op} \times \mcA \longrightarrow \Mod(\mathbb{Z})$ as follows. Let $N \in \mcA$ and $i \in \mathbb{N}$. By \cite[Prop. 1.4.13]{EnJen00} and Lemma \ref{ComparisonT}, it can be shown that there is a well defined (contravariant) functor
\[
\underline{\Ext}^i_{\mcX}(-,N) \colon \LPRes(\mcX)^{\rm op} \longrightarrow \Mod(\mathbb{Z}),
\]
where for every $M \in \LPRes(\mcX)$,
\begin{align*}
\underline{\Ext}^i_{\mcX}(M,N) & := H^i({\rm Hom}_{\mcA}(\mcX_\bullet(M),N))
\end{align*}
is the $i$-th cohomology group of the complex ${\rm Hom}_{\mcA}(\mcX_\bullet(M),N)$, that is $\underline{\Ext}^i_{\mcX}(-,N)$ is the right derived functor of $\Hom_{\mcA}(-,N)$ with respect to $\mcX$.  

The following proposition lists several properties of $\underline{\Ext}^i_{\mcX}(-,N)$. It will be useful to recall a special type of proper resolutions.

\begin{definition}
A proper left $\mathcal{X}$-resolution $\varepsilon \colon \mcX_\bullet(M) \to M$ is \textbf{admissible} if the sequence 
\[
X_1 \to X_0 \xrightarrow{\varepsilon} M \to 0
\] 
is exact. Admissible proper right resolutions are defined dually. 
\end{definition}

\begin{proposition}[properties of relative extension bifunctors]\label{ExtUL} 
Let $\mcX \subseteq \mcA$ be a class of objects in $\mcA$. Then, the following assertions hold true:
\begin{enumerate}
\item $\underline{\Ext}^{\geq 1}_{\mcX} (\mcX,- ) = 0$.

\item If $\mcZ \subseteq \LPRes(\mcX)$ is a class of objects admitting an admissible proper left $\mcX$-resolution, then there is a natural isomorphism 
\[
\underline{\Ext}^0_{\mcX}(Z,N) \cong \Hom _{\mcA}(Z,N)
\] 
for every $Z \in \mcZ$.
\end{enumerate}
Moreover, for any $\Hom_{\mcA}(\mcX,-)$-acyclic complex
\[
\mathbb{M} \equiv 0 \to M ' \to M \to M'' \to 0,
\]
the following also hold:
\begin{enumerate}
\setcounter{enumi}{2}
\item If $\mcX$ is closed under finite coproducts and $M', M'' \in \LPRes(\mcX)$, then for every object $N \in \mathcal{A}$ there is a long exact sequence of abelian groups:
\begin{align*}
{} & \underline{\Ext}^0_{\mcX}(M'',N) \rightarrowtail \underline{\Ext}^0_{\mcX}(M,N) \to \underline{\Ext}^0_{\mcX}(M',N) \to \underline{\Ext}^1_{\mcX}(M'',N) \to \cdots \\ 
\cdots \to & \mbox{ } \underline{\Ext}^i_{\mcX}(M'',N) \to \underline{\Ext}^i_{\mcX}(M,N) \to \underline{\Ext}^i_{\mcX}(M',N) \to \underline{\Ext}^{i+1}_{\mcX}(M'',N) \to \cdots.
\end{align*}

\item If $L \in \LPRes (\mcX)$, then there is a long exact sequence of abelian groups:
\begin{align*}
{} & \underline{\Ext}^0_{\mcX}(L,M') \rightarrowtail \underline{\Ext}^0_{\mcX}(L,M) \to \underline{\Ext}^0_{\mcX}(L,M'') \to \underline{\Ext}^1_{\mcX}(L,M') \to \cdots \\ 
\cdots \to & \mbox{ } \underline{\Ext}^i_{\mcX}(L,M') \to \underline{\Ext}^i_{\mcX}(L,M) \to \underline{\Ext}^i_{\mcX}(L,M'') \to \underline{\Ext}^{i+1}_{\mcX}(L,M') \to \cdots.
\end{align*}
\end{enumerate}
\end{proposition}

\begin{proof}
For part (1), it suffices to consider for each $X \in \mcX$ the proper and exact left $\mcX$-resolution $\cdots \xrightarrow{0} X \xrightarrow{{\rm id}} X  \xrightarrow{0} X \xrightarrow{{\rm id}} X \to 0$, while part (2) is straightforward from the definition of $\underline{\Ext}^0_{\mcX}(-,\sim)$. On the other hand, part (3) follows as in the proof of \cite[Thm. 8.2.3]{EnJen00}, by using Lemma \ref{HorseshoeL} (see also Avramov and Martsinkovsky's \cite[Prop. 4.6]{AvMar02}). Finally, part (4) follows by \cite[Prop. 4.4]{AvMar02}. 
\end{proof}

Dually, for $M \in \mcA$, $\mcY \subseteq \mcA$ and $i \in \mathbb{N}$ it can be shown that there is a well defined (covariant) functor
\[
\overline{\Ext}^i_{\mcY}(M,-) \colon \RPRes(\mcY) \to \Mod(\mathbb{Z}),
\]
where for any $N \in \RPRes(\mcY)$ and any proper right $\mcY$-resolution $0 \to N \to \mcY^\bullet(N)$ of $N$,
\begin{align*}
\overline{\Ext}^i_{\mcY}(M,N) & := H^i({\rm Hom}_{\mcA}(M,\mcY^\bullet(N))).
\end{align*} 
In other words, $\overline{\Ext}^i_{\mcY}(M,-)$ is the right derived functor of $\Hom_{\mcA}(M,-)$ with respect to $\mcY$, satisfying the dual properties from Proposition \ref{ExtUL}.

%%%%%%%%%%%%%%%%%%%%%%%%%%%%%%%%%%%%%
%%%%%%%%%%%%%%%%%%%%%%%%%%%%%%%%%%%%%
%%%%%%%%%%%%%%%%%%%%%%%%%%%%%%%%%%%%%
%%%%%%%%%%%%%%%%%%%%%%%%%%%%%%%%%%%%%

\section{Balance systems and induced balanced pairs}\label{sec:main}

Balanced pairs were firstly introduced by Chen in \cite{Chen}. This notion comprises the conditions that two classes $\mathcal{X}$ and $\mathcal{Y}$ of objects in an abelian category $\mathcal{A}$ need to fulfill in order to obtain balance of Hom by $\mathcal{X} \times \mathcal{Y}$. Specifically, one needs that:
\begin{itemize}
\item $\mcX$ is precovering and $\mcY$ is preenveloping;

\item every object in $\mathcal{A}$ admits a $\Hom_{\mcA}(-,\mcY)$-acyclic proper left $\mcX$-resolution;

\item every object in $\mathcal{A}$ admits a $\Hom_{\mcA}(\mcX,-)$-acyclic proper right $\mcY$-resolution.
\end{itemize}
Sometimes it is difficult to verify each of these conditions for all of the objects in the whole domain category $\mcA$ (that is, if we want to balance Hom over $\mcA \times \mcA$). It is possible to overcome this limitation by restricting balance to classes contained in $\LPRes(\mcX)$ and $\RPRes(\mcY)$, in the cases where it is not possible to approximate all the objects in $\mcA$ by objects in $\mcX$ and $\mcY$. The following definition is an adaptation of \cite[Def. 1.1]{Chen} for this purpose.

\begin{definition}\label{def:balanced_pair} 
Let $\mcX$, $\tilde{\mcX}$, $\mcY$ and $\tilde{\mcY}$ be classes of objects in an abelian category $\mcA$. The pair $(\mcX,\mcY)$ is a (\textbf{n admissible}) \textbf{balanced pair with respect to a pair $\bm{(\tilde{\mcX},\tilde{\mcY})}$} if the following conditions are satisfied: 
\begin{enumerate}
\item[$\mathsf{(bp1)}$] $\mcX \subseteq \tilde{\mcX} \subseteq \LPRes(\mcX)$ and $\mcY \subseteq \tilde{\mcY} \subseteq \RPRes(\mcY)$.

\item[$\mathsf{(bp2)}$] Every $M \in \tilde{\mcX}$ has a (n admissible) $\Hom_{\mcA}(-,\mcY)$-acyclic proper left $\mcX$-resolution.

\item[$\mathsf{(bp3)}$] Every $N \in \tilde{\mcY}$ has a (n admissible) $\Hom_{\mcA}(\mcX,-)$-acyclic proper right $\mcY$-resolution.
\end{enumerate}
\end{definition}

One of the main advantages of having balance of Hom by $\mcX \times \mcY$ over $\tilde{\mcX} \times \tilde{\mcY}$ is to have several ways to define right derived functors. For instance, in Holm's \cite[Def. 3.7]{Holm05} the Gorenstein extension groups ${\rm Gext}^i_R(M,N)$, where $M$ has finite Gorenstein projective dimension and $N$ has finite Gorenstein injective dimension, are defined as right derived functors of Hom by either taking proper left Gorenstein projective ${\rm Hom}_R(-,\mathcal{GI}(R))$-acyclic resolutions of $M$, or proper right Gorenstein injective ${\rm Hom}_R(\mathcal{GP}(R),-)$-acyclic resolutions of $N$. In a more general sense, the following result gives us equivalent ways to compute the relative derived bifunctors $\underline{\Ext}^i_{\mcX}(-,\sim)$ and $\overline{\Ext}^i_{\mcY}(-,\sim)$. The proof follows from \cite[Thm. 2.6]{Holm05}.

\begin{proposition}[relative derived functors]\label{CriterioBalance}  
Let $(\mcX,\mcY)$ be a balanced pair in $\mcA$ with respect to $(\tilde{\mcX},\tilde{\mcY})$. Then, for every $i \in \mathbb{N}$, $M \in \tilde{\mcX}$ and $N \in \tilde{\mcY}$, there is a natural isomorphism 
\[
\underline{\Ext}^i_{\mcX}(M,N) \cong \overline{\Ext}^i_{\mcY}(M,N),
\]
which yields a well defined bifunctor 
\[
\Ext^i_{(\mcX,\mcY)}(-,\sim) \colon \tilde{\mcX}^{\rm op} \times \tilde{\mcY} \longrightarrow \Mod(\mathbb{Z})
\]
by setting 
\[
\Ext^i_{(\mcX,\mcY)}(M,N) := \underline{\Ext}^i_{\mcX}(M,N)
\]
or 
\[
\Ext^i_{(\mcX,\mcY)}(M,N) := \overline{\Ext}^i_{\mcY}(M,N)
\]
for every $M \in \tilde{\mcX}$ and $N \in \tilde{\mcY}$.
\end{proposition}

\begin{remark}
From Lemma \ref{ex-xyres} and Proposition \ref{ExtUL} (2), we can note that $\underline{\Ext}^0_{\mcX}$ can be extended to $\tilde{\mcX}^{\rm op} \times \mcA$ if $\mcX$ is a generator in $\mcA$, where 
\[
\Ext^0_{(\mcX,\mcY)}(M,N) = \Hom_{\mcA}(M,N) \cong \underline{\Ext}^0_{\mcX}(M,N),
\] 
for every $M \in \tilde{\mcX}$ and $N \in \mcA$. 
\end{remark}

The rest of this section will be devoted to the concept and properties of balanced systems. These will be formed by a pair of classes of objects accompanied with relative (co)generators satisfying certain orthogonality relations under $\Ext$. Such relations are in practice easier to check that conditions $\mathsf{(bp2)}$ and $\mathsf{(bp3)}$ in Definition \ref{def:balanced_pair} above. Moreover, balanced systems will be a good source to obtain balanced pairs. Before being more specific on this, let us introduce some notation and terminology.  

Let us borrow the term ``cospan'' from category theory. In our setting, a \emph{cospan} will be a triple $(\mcX,\mcX'',\mcX')$ of classes of objects in $\mcA$ along with containments $\mcX \subseteq \mcX'' \supseteq \mcX'$. If in addition $\mcX'' \subseteq \mcY$ for a class $\mcY \subseteq \mcA$ of objects in $\mcA$, we shall say that $(\mcX,\mcX'',\mcX')$ is a \emph{cospan in} $\mcY$. 

Given a cospan $(\mcX,\mcX'',\mcX')$ in $\LPRes(\mcX)$, we define the following classes of objects in $\mcA$:
\begin{itemize}
\item $\mathsf{L}_{\mcX}(\mcX',\mcX'')$ is the class of objects $M \in \mcX''$ admitting a proper left $\mcX$-resolution $\varepsilon \colon \mcX_\bullet(M) \to M$ such that $\Ker(\varepsilon) \in \mcX'$ and $ Z_i(\mcX_\bullet(M)) \in \mcX'$ for every $i \in \mathbb{N}^\ast$. \footnote{In the notation $\mathsf{L}_{\mcX}(\mcX',\mcX'')$, $\mathsf{L}_{\mcX}$ indicates that one takes \emph{left} $\mcX$-resolutions. The class $\mcX''$ placed to the right suggests that these resolutions end at an object in $\mcX''$.} 

\item $\overline{\mathsf{L}}_\mcX(\mcX',\mcX'') := \mathsf{L}_\mcX(\mcX',\mcX'') \cup \mcX$.
\end{itemize} 
Dually, for a cospan $(\mcY,\mcY'',\mcY')$ in $\RPRes(\mcY)$, we have the classes $\mathsf{R}^{\mcY}(\mcY'',\mcY')$ and $\overline{\mathsf{R}}^{\mcY}(\mcY'',\mcY')$.

\begin{remark}
Note that if $\mcX$ and $\mcX'$ are pointed (that is, $0 \in \mcX$ and $0 \in \mcX'$), then $\overline{\mathsf{L}}_\mcX(\mcX',\mcX'') = \mathsf{L}_\mcX(\mcX',\mcX'')$.
\end{remark}

\begin{example}\label{descLX} 
Let $\mcX,\omega \subseteq \mcA$ such that $\mcX$ is pointed and closed under extensions, $\omega$ is a relative cogenerator in $\mcX$ with $\id_{\mcX}(\omega) = 0$. Then from \cite[Thm. 2.8 (b)]{BMS} we know that every $M \in \mcX^\wedge$ has a special $\mcX$-precover with kernel in $\omega^\wedge$. It follows that $(\mcX,\mcX^\wedge,\omega^\wedge)$ is a cospan in $\LPRes(\mcX)$ and that $\mathsf{L}_{\mcX}(\omega^\wedge,\mcX^\wedge) = \mcX^\wedge = \overline{\mathsf{L}}_{\mcX}(\omega^\wedge,\mcX^\wedge)$.
\end{example}

The classes $\overline{\mathsf{L}}_\mcX(\mcX',\mcX'')$ and $\overline{\mathsf{R}}^{\mcY}(\mcY'',\mcY')$ offer a possible domain to obtain balance for Hom, provided that certain conditions are fulfilled. Such conditions are comprised in the following definition.

\begin{definition}\label{balanced-couple} 
A couple $[(\mcX,\omega); (\nu,\mcY)]$ of pairs of classes of objects in $\mcA$ is a \textbf{balanced system} with respect to a couple $[(\mcX',\mcX''); (\mcY'',\mcY')]$ in $\mcA$ if the following are satisfied:
\begin{enumerate}
\item[$\mathsf{(bs1)}$] $(\mcX,\mcX'',\mcX')$ is a cospan in $\LPRes(\mcX)$.

\item[$\mathsf{(bs2)}$] $(\mcY,\mcY'',\mcY')$ is a cospan in $\RPRes(\mcY)$.

\item[$\mathsf{(bs3)}$] $\mcX$ is a generator in $\mcX''$.

\item[$\mathsf{(bs4)}$] $\mcY$ is a cogenerator in $\mcY''$.

\item[$\mathsf{(bs5)}$] $\nu$ is a relative generator in $\mcY$.

\item[$\mathsf{(bs6)}$] $\omega$ is a relative cogenerator in $\mcX$. 

\item[$\mathsf{(bs7)}$] $\Ext^1_\mcA(\mcX'',\nu) = \Ext^1_\mcA(\mcX',\mcY) = 0$. 

\item[$\mathsf{(bs8)}$] $\Ext^1_\mcA(\omega,\mcY'') = \Ext^1_\mcA(\mcX,\mcY') = 0$.\footnote{Note that the even numbered conditions are the dual of the odd numbered ones. We shall use this convention in some of the upcoming definitions.}
\end{enumerate}
\end{definition}

Conditions $\mathsf{(bs1)}$, $\mathsf{(bs3)}$, $\mathsf{(bs5)}$ and $\mathsf{(bs7)}$ are enough to guarantee that any object in $\mathsf{L}_{\mcX}(\mcX',\mcX'')$ has an exact and $\Hom_\mcA(-,\mcY)$-acyclic proper left $\mcX$-resolution with cycles in $\mcX'$. This is proved in the following result.

\begin{proposition}[induced relative balanced pairs from balanced systems]\label{tripleKer-ex} 
Let $\mcX$, $\mcX'$, $\mcX''$, $\nu$, $\mcY \subseteq \mcA$ be classes of objects satisfying conditions $\mathsf{(bs1)}$, $\mathsf{(bs3)}$, $\mathsf{(bs5)}$ and $\mathsf{(bs7)}$ in Definition \ref{balanced-couple}. For every $M \in \mcX''$, if $\varepsilon \colon \mcX_\bullet(M) \to M$ is a proper left $\mcX$-resolution with $\Ker(\varepsilon) \in \mcX'$ and $ Z_i(\mcX_\bullet(M)) \in \mcX'$ for every $i \in \mathbb{N}^\ast$, then $\varepsilon$ is exact and $\Hom_\mcA(-,\mcY)$-acyclic. 

In particular, if $[(\mcX,\omega);(\nu,\mcY)]$ is a balanced system with respect to $[(\mcX',\mcX'');(\mcY'',\mcY')]$, then $(\mcX,\mcY)$ is a balanced pair with respect to $(\overline{\mathsf{L}}_\mcX(\mcX',\mcX''),\overline{\mathsf{R}}^{\mcY}(\mcY'',\mcY'))$.
\end{proposition}

\begin{proof} 
Let $M \in \mathsf{L}_{\mcX}(\mcX',\mcX'')$ with a proper left $\mcX$-resolution $\varepsilon \colon \mcX_\bullet(M) \to M$ with $\Ker(\varepsilon) \in \mcX'$ and $Z_i(\mcX_\bullet(M)) \in \mcX'$ for every $i \in \mathbb{N}^\ast$. We show that $\varepsilon$ is exact and $\Hom_{\mcA}(-,\mcY)$-acyclic. Indeed, since $\mcX$ is a generator in $\mcX''$ and $\mcX' \subseteq \mcX''$, we have by Lemma \ref{ex-xyres} that $\varepsilon$ is exact. By the same result, it remains to show that the short exact sequences
\begin{align*}
\Ker(\varepsilon) & \stackrel{\mu_0}\rightarrowtail X_0 \twoheadrightarrow M & & \text{and} & Z_i(\mcX_\bullet(M)) & \stackrel{\mu_i}\rightarrowtail X_i \twoheadrightarrow Z_{i-1}(\mcX_\bullet(M)),
\end{align*}
are $\Hom_{\mcA}(-,\mcY)$-acyclic. So let $f \colon Z_i(\mcX_\bullet(M)) \to Y$ with $Y \in \mcY$ (the proof for the sequence involving $\Ker(\varepsilon)$ is analogous). Since $\nu$ is a relative generator in $\mcY$, there is a short exact sequence 
\[
Y' \rightarrowtail V \stackrel{g}\twoheadrightarrow Y
\]
with $V \in \nu$ and $Y' \in \mcY$. Using the condition $\Ext^1_{\mcA}(\mcX',\mcY) = 0$, there exists a morphism $f' \colon Z_i(\mcX_\bullet(M)) \to V$ such that $g \circ f' = f$. On the other hand, the condition $\Ext^1_{\mcA}(\mcX'',\nu) = 0$ implies that there exists a morphism $f'' \colon X_i \to V$ such that $f'' \circ \mu_i = f'$. Hence, $(g \circ f'') \circ \mu_i = f$, and the result follows. 

The assertion concerning balanced couples is straightforward. 
\end{proof}

\begin{corollary}[relative derived functors from balanced systems]\label{coroll:tripleKer-ex}
Let $[(\mcX,\omega);(\nu,\mcY)]$ be a balanced system with respect to $[(\mcX',\mcX'');(\mcY'',\mcY')]$. Then, the following properties hold for the bifunctor
\[
\Ext^i_{(\mcX,\mcY)}(-,\sim) \colon \overline{\mathsf{L}}_\mcX(\mcX',\mcX'')^{\rm op} \times \overline{\mathsf{R}}^{\mcY}(\mcY'',\mcY') \longrightarrow \Mod(\mathbb{Z}).
\]
\begin{enumerate}
\item $\Ext^{\geq 1}_{(\mcX,\mcY)}(\mcX,-) = 0$ and $\Ext^{\geq 1}_{(\mcX,\mcY)}(-,\mcY) = 0$. 

\item $\Ext^0_{(\mcX,\mcY)}(M,N) = \Hom_{\mcA}(M,N)$ for $M \in \overline{\mathsf{L}}_\mcX(\mcX',\mcX'')$ and $N \in \overline{\mathsf{R}}^{\mcY}(\mcY'',\mcY')$. Moreover, we can replace $N \in \overline{\mathsf{R}}^{\mcY}(\mcY'',\mcY')$ by $N \in \mcA$ if $\mcX$ is a generator in $\mcA$, and $M \in \overline{\mathsf{L}}_\mcX(\mcX',\mcX'')$ by $M \in \mcA$ if $\mcY$ is a cogenerator in $\mcA$. 
\end{enumerate}
Moreover, for any $\Hom_{\mcA}(\mcX,-)$-acyclic complex
\[
\mathbb{M} \equiv 0 \to M ' \to M \to M'' \to 0,
\]
in $\overline{\mathsf{L}}_\mcX(\mcX',\mcX'')$ (that is, $M, M', M'' \in \overline{\mathsf{L}}_\mcX(\mcX',\mcX'')$), the following also hold:
\begin{enumerate}
\setcounter{enumi}{2}
\item If $\mcX$ is closed under finite coproducts, then for every $N \in \overline{\mathsf{R}}^{\mcY}(\mcY'',\mcY')$ there is a long exact sequence of abelian groups:
\begin{align*}
{} & \Hom_{\mcA}(M'',N) \rightarrowtail \Hom_{\mcA}(M,N) \to \Hom_{\mcA}(M',N) \to \Ext^1_{(\mcX,\mcY)}(M'',N) \to \cdots \\ 
{} & \cdots \to \Ext^i_{(\mcX,\mcY)}(M,N) \to \Ext^i_{(\mcX,\mcY)}(M',N) \to \Ext^{i+1}_{(\mcX,\mcY)}(M'',N) \to \cdots. 
\end{align*}

\item If $L \in \overline{\mathsf{L}}_\mcX(\mcX',\mcX'')$, then there is a long exact sequence of abelian groups:
\begin{align*}
{} & \Hom_{\mcA}(L,M') \rightarrowtail \Hom_{\mcA}(L,M) \to \Hom_{\mcA}(L,M'') \to \Ext^{1}_{(\mcX,\mcY)}(L,M') \to \cdots \\ 
{} & \cdots \to \Ext^{i}_{(\mcX,\mcY)}(L,M) \to \Ext^{i}_{(\mcX,\mcY)}(L,M'') \to \Ext^{i+1}_{(\mcX,\mcY)}(L,M') \to \cdots. \\
\end{align*}
\end{enumerate}
\end{corollary}

\begin{example}[balanced systems in relative Gorenstein homological algebra]\label{ex:balanced_sys}
~\
\begin{enumerate}
\item Easily, we can note that $[(\mathcal{P}(R),\mathcal{P}(R));(\mathcal{I}(R),\mathcal{I}(R)]$ is a balanced system with respect to $[(\Mod(R),\Mod(R));(\Mod(R),\Mod(R))]$. The corresponding relative extension bifunctors $\Ext^i_{(\mcX,\mcY)}(-,\sim)$ on $\Mod(R)^{\rm op} \times \Mod(R)$ is this case are the usual extension functors $\Ext^i_{R}(-,\sim)$. The properties from Corollary \ref{coroll:tripleKer-ex} are well known. 

\item From {\v{S}}aroch and {\v{S}}t'ov{\'{\i}}{\v{c}}ek's \cite[Thm. 4.11]{SarochStovicek} and Estrada, Iacob and Yeomans' \cite[Thm. 1]{EstradaIacobYeomans}, we know that if $R$ is a ring such that every Gorenstein projective left $R$-module is Gorenstein flat, and if every Gorenstein flat left $R$-module has finite Gorenstein projective dimension, then $\mcGP(R)$ is special precovering in $\Mod(R)$. Examples of such rings are given in \cite[\S \ 4]{EstradaIacobYeomans}. It then follows that over such rings, we have that 
\[
\mathsf{L}_{\mcGP(R)}(\mcGP(R)^\perp,\Mod(R)) = \Mod(R) = \LPRes(\mcGP(R)).
\] 
Moreover, in \cite[Thm. 5.6]{SarochStovicek} it is also proved that over an arbitrary ring $R$, the class $\mcGI(R)$ is special preenveloping. This in turn implies that 
\[
\mathsf{R}^{\mcGI(R)}({}^\perp\mcGI(R),\Mod(R)) = \Mod(R) = \RPRes(\mcGI(R)).
\] 
Thus, if $\mathcal{GF}(R)$ denotes the class of Gorenstein flat left $R$-modules and $R$ is a ring such that $\mcGP(R) \subseteq \mathcal{GF}(R) \subseteq \mathcal{GP}(R)^\wedge$, then after setting
\begin{align*}
\hspace{0.3cm} \mcX & = \mcGP(R), & \omega & = \mcP(R), & \mcX' & = \mcGP(R)^\perp, & \mcX'' & = \Mod(R), \\
\hspace{0.3cm} \mcY & = \mcGI(R), & \nu & = \mcI(R), & \mcY' & = {}^\perp\mcGI(R), & \mcY'' & = \Mod(R). 
\end{align*}
we have that the following conditions are equivalent:
\begin{enumerate}
\item[(i)] $[(\mcGP(R),\mcP(R));(\mcI(R),\mcGI(R))]$ is a balanced system with respect to 
\[
[(\mcGP(R)^\perp,\Mod(R));(\Mod(R),{}^\perp\mcGI(R))].
\]

\item[(ii)] $\Ext^1_R(\mcGP(R)^\perp,\mcGI(R)) = 0$ and $\Ext^1_R(\mcGP(R),{}^\perp\mcGI(R)) = 0$.

\item[(iii)] $(\mathcal{GP}(R),\mathcal{GI}(R))$ is a balanced pair (with respect to $(\Mod(R),\Mod(R))$).
\end{enumerate}
Moreover, if any of the previous equivalent conditions holds, the corresponding relative bifunctors $\Ext^i_{(\mcGP(R),\mcGI(R))}(-,\sim)$ are usually denoted by ${\rm Gext}^i_R(-,\sim)$ in the literature, and they satisfy the properties described in Corollary \ref{coroll:tripleKer-ex}. Note that the equivalence between (i) and (ii) is clear from the previous comments, while (ii) $\Rightarrow$ (iii) is a consequence of Proposition \ref{tripleKer-ex}. Finally, we can deduce (iii) $\Rightarrow$ (ii) from \cite[Prop. 2.2]{Chen}. Indeed, suppose $(\mathcal{GP}(R),\mathcal{GI}(R))$ is a balanced pair and take $M \in \mathcal{GP}(R)$ and $N \in {}^\perp\mcGI(R)$. Since $\mcGI(R)$ is special preenveloping, we have a short exact sequence $N \rightarrowtail E \twoheadrightarrow N'$ where $N' \in {}^\perp\mcGI(R)$ and $E \in {}^\perp\mcGI(R) \cap \mcGI(R) \subseteq \mcI(R)$. This sequence is clearly $\Hom_R(-,\mcGI(R))$-acyclic, and so it is $\Hom_R(\mcGP(R),-)$-acyclic since $(\mcGP(R),\mcGI(R))$ is a balanced pair. It then follows from the induced exact sequence
\[
\Hom_R(M,E) \twoheadrightarrow \Hom_R(M,N') \to \Ext^1_R(M,N) \to \Ext^1_R(M,E) = 0
\]
that $\Ext^1_R(M,N) = 0$. In a similar way, $\Ext^1_R(\mcGP(R)^\perp,\mcGI(R)) = 0$. 

For rings over which every left $R$-module has finite Gorenstein projective and Gorenstein injective dimensions, special Gorenstein projective precovers (resp., special Gorenstein injective preenvelopes) can be constructed so that their kernels belong to $\mathcal{P}(R)^\wedge$ (resp., with cokernels in $\mathcal{I}(R)^\vee$). In this case, we can set $\mcX' = \mcP(R)^\wedge$ and $\mcY' = \mcI(R)^\vee$, and following the arguments in \cite[Proof of Lem. 3.4]{Holm05} we can show that $\Ext^1_R(\mcP(R)^\wedge,\mcGI(R)) = 0$ and $\Ext^1_R(\mcGP(R),\mcI(R)^\vee) = 0$. 

\item Let $\mathcal{DP}(R)$ and $\mathcal{DI}(R)$ denote the classes of \textbf{Ding projective} and \textbf{Ding injective} left $R$-modules. Recall that $\mathcal{DP}(R)$ is formed by all cycles of exact complexes of projective left $R$-modules which are $\Hom_R(-,\mathcal{F}(R))$-acyclic. Dually, $\mathcal{DI}(R)$ is defined as the collection of all cycles of exact complexes of injective left $R$-modules which are $\Hom_R(\text{FP}_1\mbox{-}\mcI(R),-)$-acyclic, where if $\mathcal{FP}_1(R)$ denotes the class of finitely presented left $R$-modules, then $\text{FP}_1\mbox{-}\mcI(R) := [\mathcal{FP}_1(R)]^{\perp_1}$ denotes the class of FP-injective (or absolutely pure) left $R$-modules. 

Following the approach of the previous example, let us set 
\begin{align*}
\hspace{0.3cm} \mcX & = \mathcal{DP}(R), & \omega & = \mcP(R), & \mcX' & = \mathcal{DP}(R)^\perp, & \mcX'' & = \Mod(R), \\
\hspace{0.3cm} \mcY & = \mathcal{DI}(R), & \nu & = \mcI(R), & \mcY' & = {}^\perp\mathcal{DI}(R), & \mcY'' & = \Mod(R). 
\end{align*}
By Iacob's \cite[Thm. 2]{IacobCoresolved} and Bravo, Gillespie and Hovey's \cite[\S \ 8]{BravoGillespieHovey}, every left $R$-module has a Ding projective special precover, provided that $R$ is a right coherent ring. Concerning Ding injective left $R$-modules, we have from Gillespie and Iacob's \cite[Thm. 44]{GillespieIacobDuality} that $\mathcal{DI}(R)$ is special preenveloping for any arbitrary ring $R$. It then follows that if $R$ is a right coherent ring, then the equivalence (i) $\Leftrightarrow$ (ii) $\Leftrightarrow$ (iii) in the previous example holds if we replace $\mcGP(R)$ by $(\mathcal{DP}(R)$, and $\mcGI(R)$ by $\mathcal{DI}(R)$, respectively. Moreover, if any of these corresponding equivalent conditions holds, the  relative bifunctors $\Ext^i_{(\mathcal{DP}(R),\mathcal{DI}(R))}(-,\sim)$ are denoted by ${\rm Dext}^i_R(-,\sim)$, following Yang's notation in \cite{Yang12}. From Corollary \ref{coroll:tripleKer-ex} and its dual, we can obtain \cite[parts (1), (2) \& (3) of Rmk. 3.7]{Yang12}\footnote{Yang proves the mentioned properties of ${\rm Dext}$ is the case where $R$ is a Ding-Chen ring.}.

In particular, one can note that $[(\mathcal{DP}(R),\mathcal{P}(R));(\mathcal{I}(R),\mathcal{DI}(R))]$ is a balanced system with respect to $[(\mathcal{P}(R)^\wedge,\Mod(R));(\Mod(R),\mathcal{I}(R)^\vee)]$ in the case where $R$ is a Ding-Chen ring. 

\item In this example we analyze Gorenstein modules relative to complete duality pairs. Let $R$ be a commutative ring. Let us recall that two classes $\mathcal{C}$ and $\mathcal{D}$ of $R$-modules form a \textbf{complete duality pair} $(\mathcal{C,D})$ provided that:
\begin{itemize}
\item $C \in \mathcal{C}$ if, and only if, $C^+ \in \mathcal{D}$ (where $C^+ := \Hom_{\mathbb{Z}}(C,\mathbb{Q / Z})$).

\item $D \in \mathcal{D}$ if, and only if, $D^+ \in \mathcal{C}$.

\item $\mathcal{C}$ and $\mathcal{D}$ are closed under direct summands, and $\mathcal{D}$ is closed under finite coproducts.
 
\item $\mathcal{C}$ is closed under arbitrary coproducts and extensions, and $R \in \mathcal{C}$. 
\end{itemize}
In Gillespie's \cite[Def. 4.1]{GillespieDualityStable}, an $R$-module is called \textbf{Gorenstein $\bm{(\mathcal{C,D})}$-injective} if it is a cycle of an exact and $\Hom_R(\mathcal{D},-)$-acyclic complex of injective $R$-modules. On the other hand, an $R$-module is called \textbf{Gorenstein $\bm{(\mathcal{C,D})}$-projective} if it is a cycle of an exact and $\Hom_R(-,\mathcal{C})$-acyclic complex of projective $R$-modules. Let us denote by $\mcGP_{(\mathcal{C,D})}(R)$ and $\mcGI_{(\mathcal{C,D})}(R)$ the classes of Gorenstein $(\mathcal{C,D})$-projective and Gorenstein $(\mathcal{C,D})$-injective $R$-modules, respectively. From these definitions and \cite[Thms. 4.6 \& 4.9]{GillespieDualityStable}, the choice
\begin{align*}
\hspace{0.6cm} \mcX & = \mcGP_{(\mathcal{C,D})}(R), & \omega & = \mcP(R), & \mcX' & = \mcGP_{(\mathcal{C,D})}(R)^\perp, & \mcX'' & = \Mod(R), \\
\hspace{0.6cm} \mcY & = \mcGI_{(\mathcal{C,D})}(R), & \nu & = \mcI(R), & \mcY' & = {}^\perp\mcGI_{(\mathcal{C,D})}(R), & \mcY'' & = \Mod(R). 
\end{align*}
satisfies conditions from $\mathsf{(bs1)}$ to $\mathsf{(bs6)}$ in Definition \ref{balanced-couple}, and also we can get the equivalence (i) $\Leftrightarrow$ (ii) $\Leftrightarrow$ (iii) from the second example after replacing $\mcGP(R)$ by $\mcGP_{(\mathcal{C,D})}(R)$, and $\mcGI(R)$ by $\mcGI_{(\mathcal{C,D})}(R)$, respectively.

Let us now focus on the particular case where $\mathcal{C}$ is the class of level $R$-modules and $\mathcal{D}$ is the class of absolutely clean $R$-modules. Let us recall that an $R$-module is \textbf{of type $\bm{\text{FP}_\infty}$} (or \textbf{super finitely presented}) if it admits an exact left resolution by finitely generated projective $R$-modules. If we denote this class by $\mathcal{FP}_\infty(R)$, then the orthogonal complement $\mathcal{A} = \mathcal{FP}_\infty(R)^{\perp_1}$ is the class of \textbf{absolutely clean} $R$-modules. On the other hand, an $R$-module $L$ is \textbf{level} if ${\rm Tor}^R_1(F,L) = 0$ for every $F \in \mathcal{FP}_\infty(R)$. It is known that level and absolutely clean $R$-modules form a complete duality pair (see Bravo, Gillespie and Hovey's \cite[\S \ 2]{BravoGillespieHovey}), and the Gorenstein projective and Gorenstein injective $R$-modules relative to this pair, known as \textbf{Gorenstein AC-projective} and \textbf{Gorenstein AC-injective}, were introduced and widely studied also in \cite{BravoGillespieHovey}. Let us denote these classes of Gorenstein modules by $\mcGP_{\rm AC}(R)$ and $\mcGI_{\rm AC}(R)$, respectively. In the case where $[(\mcGP_{\rm AC}(R),\mcP(R));(\mcI(R),\mcGI_{\rm AC}(R))]$ is a balanced system with respect to $[(\mcGP_{\rm AC}(R)^\perp,\Mod(R));(\Mod(R),{}^\perp\mcGI_{\rm AC}(R))]$, we denote the corresponding relative extension bifunctors by ${\rm G}_{\rm ac}{\rm ext}^i_R(-,\sim)$. We are not aware whether the existence and properties of these bifunctors have been documented before, so we mention that:
\begin{itemize}
\item ${\rm G}_{\rm ac}{\rm ext}^{\geq 1}_R(\mcGP_{\rm AC}(R),-) = 0$ and ${\rm G}_{\rm ac}{\rm ext}^{\geq 1}_R(-,\mcGI_{\rm AC}(R)) = 0$. 

\item ${\rm G}_{\rm ac}{\rm ext}^{0}_R(M,N) = \Hom_R(M,N)$ for every $M, N \in \Mod(R)$. 

\item If $0 \to M ' \to M \to M'' \to 0$ is a $\Hom_R(\mcGP_{\rm AC}(R),-)$-acyclic or $\Hom_R(-,\mcGI_{\rm AC}(R))$-acyclic (and so also exact) complex of $R$-modules, then for every $N \in \Mod(R)$ there are long exact sequences of abelian groups 
\begin{align*}
{} & \Hom_R(M'',N) \rightarrowtail \Hom_R(M,N) \to \Hom_R(M',N) \to {\rm G}_{\rm ac}{\rm ext}^1_R(M'',N) \to \cdots \\ 
{} & \cdots \to {\rm G}_{\rm ac}{\rm ext}^i_R(M,N) \to {\rm G}_{\rm ac}{\rm ext}^i_R(M',N) \to {\rm G}_{\rm ac}{\rm ext}^{i+1}_R(M'',N) \to \cdots, \\
{} & \Hom_R(N,M') \rightarrowtail \Hom_R(N,M) \to \Hom_R(N,M'') \to {\rm G}_{\rm ac}{\rm ext}^1_R(N,M') \to \cdots \\ 
{} & \cdots \to {\rm G}_{\rm ac}{\rm ext}^i_R(N,M) \to {\rm G}_{\rm ac}{\rm ext}^i_R(N,M'') \to {\rm G}_{\rm ac}{\rm ext}^{i+1}_R(N,M') \to \cdots. 
\end{align*}
\end{itemize}
One particular example in which $(\mcGP_{\rm AC}(R),\mcGI_{\rm AC}(R))$ is a balanced pair occurs when $R$ is an AC-Gorenstein ring (see Gillespie's \cite[Def. 4.1]{GillespieAC}). Over such rings, $\mcGP_{\rm AC}(R)^\perp$ and ${}^\perp\mcGI_{\rm AC}(R)$ coincide with the class of $R$-modules with finite absolutely clean dimension (or equivalently, finite level dimension \cite[Thm. 4.2]{GillespieAC}), and so assertion (ii) holds.
\end{enumerate}
The previous classes of relative Gorenstein projective and injective modules also constitute sources of balance over arbitrary rings, but restricted to certain subcategories of $\Mod(R)$, as we will mention below in Example \ref{ex:finite_balanced_systems}. 
\end{example}

Among the balanced systems considered in this work, we shall focus on a special type that we call \emph{finite}, in the sense that the induced balance for Hom, in the sense of Proposition \ref{tripleKer-ex}, is obtained over classes with finite relative left and right resolution dimensions, instead of $\overline{\mathsf{L}}_\mcX(\mcX',\mcX'') \times \overline{\mathsf{R}}^{\mcY}(\mcY'',\mcY')$.

\begin{definition}\label{finite-balanced-system} 
We shall say that two pairs $(\mcX,\omega)$ and $(\nu,\mcY)$ of classes of objects in $\mcA$ form a \textbf{finite balanced system}, denoted $[(\mcX,\omega);(\nu,\mcY)]$, provided that:
\begin{enumerate}
\item[$\mathsf{(fbs1)}$] $\mcX$ is pointed and closed under extensions. 

\item[$\mathsf{(fbs2)}$] $\mcY$ is pointed and closed under extensions. 

\item[$\mathsf{(fbs3)}$] $\omega$ is a relative cogenerator in $\mcX$ with $\id_{\mcX}(\omega) = 0$.

\item[$\mathsf{(fbs4)}$] $\nu$ is a relative generator in $\mcY$ with $\pd_{\mcY}(\nu) = 0$. 

\item[$\mathsf{(fbs5)}$] $\Ext^1_\mcA(\mcX^\wedge,\nu) = \Ext^1_\mcA(\omega^\wedge,\mcY) = 0$.

\item[$\mathsf{(fbs6)}$] $\Ext^1_\mcA(\omega,\mcY^\vee) = \Ext^1_\mcA(\mcX,\nu^\vee) = 0$.
\end{enumerate}
If in addition, $\pd_{\mcY}(\omega) = 0$ and $\id_{\mcX}(\nu) = 0$, we shall say that the finite balanced system $[(\mcX,\omega);(\nu,\mcY)]$ is \textbf{strong}. 
\end{definition}

Most of our examples of finite balanced systems in the following sections will be strong, although this additional feature is not needed to obtain relative balance, as described in the following result.

\begin{proposition}[relative balanced pairs from finite balanced systems]\label{Tfbs} 
Given a finite balanced system $[(\mcX,\omega);(\nu,\mcY)]$ in $\mcA$, the pair $(\mcX,\mcY)$ is balanced with respect to $(\mcX^\wedge, \mcY^\vee)$.
\end{proposition}

\begin{proof} 
By Example \ref{descLX} and its dual, we have that $[(\mcX,\omega);(\nu,\mcY)]$ is balanced with respect to $[(\omega^\wedge,\mcX^\wedge); (\mcY^\vee,\nu^\vee)]$, and also $\overline{\mathsf{L}}_\mcX(\omega^\wedge,\mcX^\wedge) = \mcX^\wedge$ and $\overline{\mathsf{R}}^{\mcY}(\mcY^\vee,\nu^\vee)=\mcY^\vee$. Thus, the result follows by Proposition \ref{tripleKer-ex}.
\end{proof}

\begin{remark}\label{rem:sfbs}
If $[(\mcX,\omega);(\nu,\mcY)]$ is a strong finite balanced system, then by dimension shifting we have that 
\[
\pd_{\mcY}(\omega^\wedge) = \id_{\mcX}(\nu^\vee) = \id_{\omega}(\mcY^\vee) = \pd_{\nu}(\mcX^\wedge) = 0.
\] 
\end{remark}

\begin{example}\label{ex:finite_balanced_systems}
Over an arbitrary ring $R$, 
\begin{align*}
& [(\mathcal{GP}(R),\mathcal{P}(R));(\mathcal{I}(R),\mathcal{GI}(R))], \\ 
& [(\mathcal{DP}(R),\mathcal{P}(R));(\mathcal{I}(R),\mathcal{DI}(R))] \text{ \ and} \\ 
& [(\mathcal{GP}_{(\mathcal{L,A})}(R),\mathcal{P}(R));(\mathcal{I}(R),\mathcal{GI}_{(\mathcal{L,A})}(R))]
\end{align*}
are strong finite balanced systems. In particular, the equalities $\Ext^1_R(\omega^\wedge,\mcY) = 0$ and $\Ext^1_R(\mcX,\nu^\vee) = 0$ in $\mathsf{(fbs5)}$ and $\mathsf{(fbs6)}$ follow using the arguments from \cite[Proof of Lem. 3.4]{Holm05}. The corresponding derived functors ${\rm Gext}^i_R(-,\sim)$, ${\rm Dext}^i_R(-,\sim)$ and ${\rm G}_{\rm ac}{\rm ext}^i_R(-,\sim)$ are defined on the subcategories 
\[
(\mathcal{GP}(R)^\wedge)^{\rm op} \times \mathcal{GI}(R)^\vee, \ (\mathcal{DP}(R)^\wedge)^{\rm op} \times \mathcal{DI}(R)^\vee \ \text{and} \ (\mathcal{GP}_{\rm AC}(R)^\wedge)^{\rm op} \times \mathcal{GI}_{\rm AC}(R)^\vee,
\] 
respectively. Note that in these examples, no extra assumptions on the ring $R$ are needed in order to obtain balance, compared with Example \ref{ex:balanced_sys}. The cost of this is that this balance is restricted to the mentioned subcategories. 
\end{example}

Below we provide a characterization of strong finite balanced systems.

\begin{lemma}\label{tripleKer-ex2} 
Let $\mcX$, $\mcX'$, $\mcY$ and $\nu$ be classes of objects in $\mcA$ satisfying the following:
\begin{enumerate}
\item $\mcX$ is a relative quasi-generator in $\mcX'$.

\item $\nu$ is a relative generator in $\mcY$.

\item $\Ext^1_\mcA(\mcX',\nu) = 0$ and $\id_\mcX(\nu) = 0$.

\item $\pd_\mcY(L) < \infty$ for every $L \in \mcX'$. 
\end{enumerate}
Then $\id_{\mcX'}(\mcY) = 0$.
\end{lemma}

\begin{proof} 
We assert that $\id_{\mcX'}(\nu) = 0$. Indeed, let $L \in \mcX'$ and $V \in \nu$. Since $\mcX$ is a relative quasi-generator in $\mcX',$ there is an exact left $\mcX$-resolution $\varepsilon \colon \mathcal{X}_\bullet(L) \to L$ of $L$ such that ${\rm Ker}(\varepsilon) \in \mcX'$ and $Z_i(\mathcal{X}_\bullet(L)) \in \mcX'$ for every $i \in \mathbb{N}^\ast$. Since $\id_\mcX(\nu) = 0$, by dimension shifting we have that $\Ext^{2}_\mcA(L,V) \cong \Ext^1_\mcA({\rm Ker}(\varepsilon),V) = 0$ and $\Ext^{i+2}_\mcA(L,V) \cong \Ext^1_\mcA(Z_{i}(\mathcal{X}_\bullet(L)),V) = 0$ for every $i \in \mathbb{N}^\ast$.

Now let $Y \in \mcY$. By (4), we know that $n := \pd_\mcY(L) < \infty$ is finite. On the other hand, using that $\nu$ is a relative generator in $\mcY$, we can construct an exact sequence 
\[
Y_n \rightarrowtail V_{n-1} \to \cdots \to V_1 \to V_0 \twoheadrightarrow Y
\]
where $Y_n \in \mcY$ and $V_i \in \nu$ for every $i \in \{ 0, 1, \dots, n-1 \}$. From $\id_{\mcX'}(\nu) = 0$, it follows that $\Ext^i_\mcA(L,Y) \cong \Ext^{i+n}_{\mcA}(L,Y_n) = 0$ for every $i \in \mathbb{N}^\ast$. Hence, $\id_{\mcX'}(\mcY) = 0$.
\end{proof}

\begin{proposition}[characterization of strong finite balanced systems]\label{prop:characterization_fbs}
Let $\mcX$, $\mcY$, $\omega$ and $\nu$ be classes of objects in $\mcA$ satisfying conditions from $\mathsf{(fbs1)}$ to $\mathsf{(fbs4)}$ in Definition \ref{finite-balanced-system}. Then, $[(\mcX,\omega);(\nu,\mcY)]$ is a strong finite balanced system if, and only if, the following conditions are satisfied:
\begin{enumerate}
\item $\Ext^1_{\mcA}(\omega^\wedge,\nu) = \Ext^2_{\mcA}(\mcX,\nu) = 0$.

\item $\Ext^1_{\mcA}(\omega,\nu^\vee) = \Ext^2_{\mcA}(\omega,\mcY) = 0$.

\item $\pd_{\mcY}(M) < \infty$ for every $M \in \omega^\wedge$.

\item $\id_{\mcX}(N) < \infty$ for every $N \in \nu^\vee$.

\item $\id_{\omega}(\nu) = 0$. 
\end{enumerate}
\end{proposition}

\begin{proof}
The ``only if'' part is clearly a consequence of Definition \ref{finite-balanced-system} and Remark \ref{rem:sfbs}. For the ``if'' part, suppose that conditions (1) to (5) in the previous statement hold. It is only left to show condition $\mathsf{(fbs5)}$ in Definition \ref{finite-balanced-system} and that $\pd_{\mcY}(\omega) = 0$. By \cite[Thm. 2.8]{BMS}, for every $M \in \mcX^\wedge$ there is a short exact sequence $M \rightarrowtail H \twoheadrightarrow X$ with $X \in \mcX$ and $H \in \omega^\wedge$. Then for every $V \in \nu$ we have the exact sequence 
\[
\Ext^1_{\mcA}(H,V) \to \Ext^1_{\mcA}(M,V) \to \Ext^2_{\mcA}(X,V),
\] 
where $\Ext^1_{\mcA}(H,V) = 0$ and $\Ext^2_{\mcA}(X,V) = 0$ by condition (1). Then, $\Ext^1_{\mcA}(M,V) = 0$ for every $M \in \mcX^\wedge$ and $V \in \nu$. Conditions $\Ext^1_{\mcA}(\omega^\wedge,\mcY) = 0$ and $\pd_{\mcY}(\omega) = 0$ in $\mathsf{(fbs5)}$ follows by setting $\mcX' := \omega^\wedge$ and $\mcX := \omega$ in Lemma \ref{tripleKer-ex2}.
\end{proof}

%%%%%%%%%%%%%%%%%%%%%%%%%%%%%%%%%%%%%
%%%%%%%%%%%%%%%%%%%%%%%%%%%%%%%%%%%%%
%%%%%%%%%%%%%%%%%%%%%%%%%%%%%%%%%%%%%
%%%%%%%%%%%%%%%%%%%%%%%%%%%%%%%%%%%%%

\section{Applications and examples}\label{sec:applications}

Let us apply the previous results to obtain some relative balanced pairs from balanced systems. Our examples bellow range over settings that include relative Gorenstein modules and chain complexes, flat and cotorsion quasi-coherent sheaves, among others.

%%%%%%%%%%%%%%%%%%%%%%%%%%%%%%%%%%%%%%%%%%%%%%%%
%%%%%%%%%%%%%%%%%%%%%%%%%%%%%%%%%%%%%%%%%%%%%%%%
%%%%%%%%%%%%%%%%%%%%%%%%%%%%%%%%%%%%%%%%%%%%%%%%
%%%%%%%%%%%%%%%%%%%%%%%%%%%%%%%%%%%%%%%%%%%%%%%%

\subsection*{Balance systems from cotorsion pairs sharing the same kernel}

Recall that two classes $\mcX$ and $\mcY$ of objects in an abelian category $\mcA$ form a \emph{cotorsion pair} $(\mcX,\mcY)$ in $\mcA$ if $\mcX = {}^{\perp_1}\mcY$ and $\mcY = \mcX^{\perp_1}$. Moreover, $(\mcX,\mcY)$ is \emph{complete} if every object of $\mcA$ admits a special $\mcX$-precover and a special $\mcY$-preenvelope. If $\id_{\mcX}(\mcY) = 0$, we say that $(\mcX,\mcY)$ is \emph{hereditary}.

In this section, we let $(\mathcal{X}_1,\mathcal{Y}_1)$ and $(\mathcal{X}_2,\mathcal{Y}_2)$ be  hereditary complete cotorsion pairs in an abelian category $\mathcal{A}$, satisfying 
\begin{align}\label{la_igualdad}
\mathcal{X}_1 \cap \mathcal{Y}_1 & = \mathcal{X}_2 \cap \mathcal{Y}_2.
\end{align} 
Let us refer to the previous intersection as $\omega$. We provide necessary and sufficient conditions so that $[(\mathcal{X}_1,\omega);(\omega,\mathcal{Y}_1)]$ (or equivalently, $[(\mathcal{X}_2,\omega);(\omega,\mathcal{Y}_2)]$) is a strong finite balanced system in $\mathcal{A}$. 

It is clear that conditions from $\mathsf{(fbs1)}$ to $\mathsf{(fbs4)}$ in Definition \ref{finite-balanced-system} hold for both $[(\mathcal{X}_1,\omega);(\omega,\mathcal{Y}_1)]$ and $[(\mathcal{X}_2,\omega);(\omega,\mathcal{Y}_2)]$. Moreover, condition (5) in Proposition \ref{prop:characterization_fbs} is trivial for $\omega$. Regarding (3) and (4) in the same statement, for every $M \in \omega^\wedge$ it can be shown by dimension shifting on $\resdim_{\omega}(M)$ that 
\[
\max\{ \pd_{\mcY_1}(M), \pd_{\mcY_2}(M) \} \leq \resdim_{\omega}(M),
\] 
while 
\[
\max\{ \id_{\mcX_1}(N), \id_{\mcX_2}(N) \} \leq \coresdim_{\omega}(N)
\] 
is dual for every $N \in \omega^\vee$. Moreover, we already know that 
\[
\Ext^2_\mathcal{A}(\mcX_1 \cup \mcX_2,\omega) = 0 = \Ext^2_\mathcal{A}(\omega,\mcY_1 \cup \mcY_2).
\]
The remaining conditions in Proposition \ref{prop:characterization_fbs} are not necessarily true, namely ${\rm Ext}^1_{\mcA}(\omega,\omega^\vee) = 0$ and ${\rm Ext}^1_{\mathcal{A}}(\omega^\wedge,\omega) = 0$. Thus, we have the following equivalence.

\begin{proposition}[simultaneous balance from cotorsion pairs sharing the same kernel]\label{prop:cotorsion_balance}
Let $(\mathcal{X}_1,\mathcal{Y}_1)$ and $(\mathcal{X}_2,\mathcal{Y}_2)$ be hereditary complete cotorsion pairs in an abelian category $\mathcal{A}$ which satisfy \eqref{la_igualdad}. Then, the following conditions are equivalent:
\begin{enumerate}
\item[(a)] $[(\mathcal{X}_1,\omega);(\omega,\mathcal{Y}_1)]$ is a strong finite balanced system.

\item[(b)] $[(\mathcal{X}_2,\omega);(\omega,\mathcal{Y}_2)]$ is a strong finite balanced system.

\item[(c)] $\Ext^1_\mathcal{A}(\omega,\omega^\vee) = 0$ and $\Ext^1_\mathcal{A}(\omega^\wedge,\omega) = 0$.
\end{enumerate}
\end{proposition}

\begin{example}
Let $X$ be a semi-separated noetherian scheme, and $\Qcoh(X)$ denote the category of quasi-coherent sheaves over $X$. Consider the classes $\mfF(X)$ and $\mfGF(X)$ of flat and Gorenstein flat quasi-coherent sheaves, along with their orthogonal complements $\mfC(X) = (\mfF(X))^{\perp_1}$ and $\mathfrak{GC}(X) = (\mfGF(X))^{\perp_1}$, known as cotorsion and Gorenstein cotorsion quasi-coherent sheaves. It is known from Christensen, Estrada and Thompson's work \cite[Thm. 2.2., Lem. 2.3 \& Rmk. 2.4]{CET} that $(\mfGF(X),\mfGC(X))$ and $(\mfF(X),\mfC(X))$ are hereditary complete cotorsion pairs in $\Qcoh(X)$ with 
\[
\mfGF(X) \cap \mfGC(X) = \mfF(X) \cap \mfC(X).
\] 
Then, one has that the pair $[(\mathfrak{GF}(X),\mathfrak{F}(X) \cap \mathfrak{C}(X));(\mathfrak{F}(X) \cap \mathfrak{C}(X),\mathfrak{GC}(X))]$ (or equivalently, $[(\mathfrak{F}(X),\mathfrak{F}(X) \cap \mathfrak{C}(X));(\mathfrak{F}(X) \cap \mathfrak{C}(X),\mathfrak{C}(X))]$) is a strong finite balanced system in $\mathfrak{Qcoh}(X)$ if, and only if, (c) in Proposition \ref{prop:cotorsion_balance} holds for  flat-cotorsion sheaves. This condition is difficult to fulfill in general. Indeed, there are wide classes of rings and schemes over which it is not possible to obtain balanced from the class of flat objects. It is known for instance that if $R$ is a left noetherian but not artinian ring then there is no preenveloping class $\mathcal{Q}$ such that $\Hom_R(-,\sim)$ is right balanced by $(\mathcal{F}(R), \mathcal{Q})$ (see Enochs' \cite[Thm. 4.1]{EnochsBalanceFlat}). Moreover, if $R$ is left noetherian, then $\mathcal{F}(R)$ is the left part of a balanced pair if, and only if, $R$ is left perfect (see \cite[Thm. 5.2]{EPZ}). Concerning schemes \cite[Coroll. 5.3]{EPZ}, for any noetherian and semi-separated scheme $X$, with semi-separating open affine covering $\mathcal{U} = \{U_1, \dots, U_n\}$ such that $\mathcal{O}_X(U_i)$ is a noetherian but not artinian ring for some $i \in \{1, \dots, n\}$, then $\mathfrak{F}(X)$ is not the left part of a balanced pair in $\Qcoh(X)$. It then follows from Propositions \ref{Tfbs} and \ref{prop:cotorsion_balance}, that over such $X$ one has that 
\[
\Ext^1_X(\mathfrak{F}(X) \cap \mathfrak{C}(X),(\mathfrak{F}(X) \cap \mathfrak{C}(X))^\vee) \neq 0
\] 
or 
\[
\Ext^1_X((\mathfrak{F}(X) \cap \mathfrak{C}(X))^\wedge,\mathfrak{F}(X) \cap \mathfrak{C}(X)) \neq 0.
\]
\end{example}

%%%%%%%%%%%%%%%%%%%%%%%%%%%%%%%%%%%%%
%%%%%%%%%%%%%%%%%%%%%%%%%%%%%%%%%%%%%

\subsection*{Induced balance in the category of chain complexes}

So far we have mostly presented examples of balanced systems in the category $\Mod(R)$, where the cogenerating class is $\omega = \mcP(R)$ and the generating class is $\nu = \mcI(R)$. In that follows, we explore some cases where this situation is different. 

Let us commence recalling some concepts and results for the category $\Ch(\mcA)$ of chain complexes over $\mcA$. For $n \in \mathbb{Z}$ and a complex $X_\bullet \in \Ch(\mcA)$, let $X_{\bullet}[n]$ denote the \emph{$n$-th suspension complex} defined by $X[n]_{m} := X _{m-n}$ and with differentials $\partial_m^{X_{\bullet}[n]} := (-1)^n d^{m-n}_{X_{\bullet}}$. For each $M \in \mcA$, let $\overline{M}$ denote the \emph{disk complex} 
\[
\overline{M} = \cdots \to 0 \to M \xrightarrow{{\rm id}_M} M \to 0 \to \cdots ,
\] 
with all terms equal to $0$ except $M$ in degrees $0$ and $1$. The \emph{sphere complex} at $M$, on the other hand, is the complex
\[
\underline{M} = \cdots \to 0 \to M \to 0 \to \cdots,
\] 
with all terms equal to $0$ except $M$ in degree $0$. For a pair of complexes $X_{\bullet}, Y_{\bullet} \in \Ch(\mcA)$, we denote by $\Hom_{\Ch}(X_{\bullet},Y_{\bullet})$ the abelian group of morphisms from $X_{\bullet} \to Y_{\bullet}$ in $\Ch(\mcA)$, and by $\Ext^i_{\Ch}(X_{\bullet},Y_{\bullet})$ for $i \in \mathbb{N}^\ast$ the group of $i$-fold extensions. It will be useful to recall from Gillespie's \cite[Lem. 3.1]{GillespieFlat} that there is a natural isomorphism
\begin{align}
\Ext^1_{\mcA}(X_0,M) & \cong \Ext^1_{\Ch}(X_{\bullet},\overline{M}), \label{NatExtCh1} 
\end{align}
for every $X_\bullet \in \Ch(\mcA)$ and $M \in \mcA$. 

Given $\mcX \subseteq \mcA$, we denote by $\widetilde{\mcX}$ the class of exact complexes $X_{\bullet} \in \Ch(\mcA)$ such that $Z_m(X_\bullet) \in \mcX$ for every $m \in \mathbb{Z}$, and by $\Ch(\mcX)$ the class of complexes $X_{\bullet} \in \Ch(\mcA)$ such that $X_m \in \mcX$ for every $m \in \mathbb{Z}$. For example, it is known that $\widetilde{\mcP(\mcA)}$ and $\widetilde{\mcI(\mcA)}$ are the classes of projective and injective complexes over $\mcA$. 

We point out some properties involving the previous notations. See Liang, Ding and Yang's \cite[Lem. 4.1 \& Thm. 4.6]{LiangDingYang} for a proof.

\begin{proposition}\label{rem:complexes} 
Let $\mcX$ and $\omega$ be classes of objects in $\mcA$. Then following assertions hold:
\begin{enumerate}

\item If $\Ext_{\mathcal{A}}^1(\mcX,\mcX) = 0$, then every complex $X_{\bullet} \in \widetilde{\mcX}$ is isomorphic to a coproduct\footnote{In this case, such a coproduct is also a direct product.} of complexes of the form $\overline{X}[n]$ with $X \in \mcX$ and $n \in \mathbb{Z}$. 

\item Suppose that $\mcA$ has enough injective objects, $\mcX$ is closed under extensions and direct summands, and $\omega$ is closed under finite coproducts. Then, $\omega$ is a relative cogenerator in $\mcX$ with $\id_{\mcX}(\omega) = 0$ if, and only if, $\widetilde{\omega}$ is a relative cogenerator in $\Ch(\mcX)$ with $\id_{\Ch(\mcX)}(\widetilde{\omega}) = 0$. 
\end{enumerate}
\end{proposition}

The following two results will be very useful to induce a strong finite balanced system in $\Ch(\mcA)$ from a strong finite balanced system in $\mcA$.

\begin{lemma}\label{BalanceComplejo0}
Let $\mcX$ and $\omega$ be classes of objects in $\mcA$ such that $\mcX$ is pointed and closed under extensions, and $\omega$ is a relative quasi-cogenerator in $\mcX$. If $\nu$ is a class of objects in $\mcA$ with $\Ext^1_{\mcA}(\nu,\nu) = \Ext^1_{\mcA}(\omega^{\wedge},\nu) = \Ext^2_{\mcA}(\mcX, \nu) = 0$, then 
\[
\Ext^1_{\Ch}(\Ch(\mcX^{\wedge}),\widetilde{\nu}) = 0.
\] 
In particular, $\Ext^1_{\Ch}(\Ch(\mcX)^{\wedge},\widetilde{\nu}) = 0$.
\end{lemma}

\begin{proof}
From \cite[Thm. 2.8]{BMS}, we can deduce that $\Ext_{\mcA}^1(\mcX^{\wedge}, \nu) = 0$ (see the proof of Proposition \ref{prop:characterization_fbs}). The rest follows by Proposition \ref{rem:complexes}-(1) and \eqref{NatExtCh1}. For the last assertion, note that the containment $\Ch(\mcX)^\wedge \subseteq \Ch(\mcX^\wedge)$ is trivial. 
\end{proof}

\begin{lemma}\label{BalanceComplejo} 
Let $\mcA$ be an abelian category with enough injective objects. Let $\mcX$, $\omega$ and $\nu$ be classes of objects in $\mcA$ satisfying the following conditions:
\begin{enumerate}
\item $\mcX$ is closed under extensions and direct summands.

\item $\omega$ is closed under finite coproducts and a relative cogenerator in $\mcX$.

\item $\Ext^1_{\mcA}(\nu,\nu) = 0$.

\item $\Ext^1_{\mcA}(\omega,\nu^\vee) = 0$. 

\item $\id_{\mcX}(\omega \cup \nu) = 0$. 
\end{enumerate} 
Then $\id_{\Ch(\mcX)}(\widetilde{\nu}^\vee) = 0$.
\end{lemma}

\begin{proof}
Let $X_\bullet \in \Ch(\mcX)$ and $D_\bullet \in \widetilde{\nu}^\vee$ with $m = \coresdim_{\widetilde{\nu}}(D_\bullet) < \infty$. We prove that $\Ext^{\geq 1}_{\Ch}(X_\bullet,D_\bullet) = 0$ by induction on $m$. The case $m = 0$ can be deduced from Proposition \ref{rem:complexes}-(1), \eqref{NatExtCh1} and the fact that $\mcA$ has enough injectives. Indeed, $D_{\bullet} \in \widetilde{\nu}$ can be written as a direct product of complexes of the form $\overline{V}[n]$ with $V \in \nu$ and $n \in \mathbb{Z}$. It suffices to show that $\Ext^{\geq 1}_{\mcA}(X_\bullet,\overline{V}) = 0$. Note that every injective $i$-th cosyzygy of $\overline{V}$ is a disk $\overline{K}$ with $K$ an injective $i$-th cosyzygy of $V$, and thus
\[
\Ext^{i+1}_{\Ch}(X_\bullet,\overline{V}) \cong \Ext^1_{\Ch}(X_\bullet,\overline{K}) \cong \Ext^1_{\mathcal{A}}(X_0,K) \cong \Ext^{i+1}_{\mcA}(X_0,V) = 0.
\]

Now suppose that $\Ext^{\geq 1}_{\Ch}(X_\bullet,D'_\bullet) = 0$ for every $D'_\bullet \in \widetilde{\nu}^\vee$ with $\coresdim_{\widetilde{\nu}}(D'_\bullet) \leq m-1$. Note that we can form an exact sequence $D_\bullet \rightarrowtail V_\bullet \twoheadrightarrow D'_\bullet$ with $V_\bullet \in \widetilde{\nu}$ and $\coresdim_{\widetilde{\nu}}(D'_\bullet) \leq m-1$. On the one hand, for every $i \in \mathbb{N}^\ast$ we have an exact sequence 
\[
\Ext^i_{\Ch}(X_\bullet,D'_\bullet) \to \Ext^{i+1}_{\Ch}(X_\bullet,D_\bullet) \to \Ext^{i+1}_{\Ch}(X_\bullet,V_\bullet)
\]
where $\Ext^{i+1}_{\Ch}(X_\bullet,V_\bullet) = 0$ by the case $m = 0$, and $\Ext^i_{\Ch}(X_\bullet,D'_\bullet) = 0$ by the induction hypothesis. Thus, $\Ext^{\geq 2}_{\Ch}(X_\bullet,D_\bullet) = 0$. On the other hand, by Proposition \ref{rem:complexes}-(2) we have an exact sequence $X_\bullet \rightarrowtail W_\bullet \twoheadrightarrow X'_\bullet$ with $W_\bullet \in \widetilde{\omega}$ and $X'_\bullet \in \Ch(\mcX)$. Then, there is an exact sequence 
\[
\Ext^1_{\Ch}(W_\bullet,D_\bullet) \to \Ext^1_{\Ch}(X_\bullet,D_\bullet) \to \Ext^2_{\Ch}(X'_\bullet,D_\bullet)
\] 
where $\Ext^2_{\Ch}(X'_\bullet,D_\bullet) = 0$ and $\Ext^1_{\Ch}(W_\bullet,D_\bullet) = 0$. For the latter, note that since $\Ext^1_{\mcA}(\omega,\omega) = 0$, one can write $W_\bullet$ as a coproduct of disk complexes centered at objects in $\omega$. So it suffices to note that $\Ext^1_{\Ch}(\overline{W},D_\bullet) \cong \Ext^1_{\Ch}(W,D_0) = 0$ for every $W \in \omega$, but this follows by the assumption that $\Ext^1_{\mcA}(\omega,\nu^\vee) = 0$. Indeed, $D_0 \in \nu^\vee$ is a consequence of condition (3).
\end{proof}

The following result is a direct consequence of Propositions \ref{Tfbs}, \ref{prop:characterization_fbs} and \ref{rem:complexes}-(2), and Lemmas \ref{BalanceComplejo0} and \ref{BalanceComplejo} (and their duals).

\begin{proposition}[induced strong balanced systems in chain complexes]
Suppose $[(\mcX, \omega);(\nu, \mcY)]$ is a strong finite balanced system in an abelian category $\mcA$ with enough projective and injective objects, such that $\mcX$ and $\mcY$ are closed under direct summands, and $\omega$ and $\nu$ under finite coproducts. Then, $[(\Ch(\mcX),\widetilde{\omega});(\widetilde{\nu},\Ch(\mcY))]$ is a strong finite balanced system in $\Ch(\mcA)$. Furthermore, $(\Ch(\mcX),\Ch(\mcY))$ is a balanced pair with respect to $(\Ch(\mcX)^\wedge,\Ch(\mcY)^\vee)$.
\end{proposition}

\begin{example}
The balance situations we have mentioned in previous examples concerning Gorenstein homological algebra of modules carry over to chain complexes. From Example \ref{ex:finite_balanced_systems}, we know that $[(\mathcal{GP}(R),\mathcal{P}(R));(\mathcal{I}(R),\mathcal{GI}(R))]$ is a strong finite balanced system in $\Mod(R)$. Then by the previous proposition, we obtain the strong finite balanced system in $\Ch(R)$ given by 
\[
[(\Ch(\mathcal{GP}(R)),\widetilde{\mathcal{P}(R)});(\widetilde{\mathcal{I}(R)},\Ch(\mathcal{GI}(R)))].
\]
On the other hand, it is known from Yang and Liu's \cite[Thm. 2.2 \& Prop. 2.8]{YangLiu11} that over an arbitrary ring $R$, $\Ch(\mathcal{GP}(R))$ and $\Ch(\mathcal{GI}(R))$ are precisely the classes of Gorenstein projective and Gorenstein injective complexes. 

One also has a similar example of balance in $\Ch(R)$ by the classes of Ding projective and Ding injective complexes, since $\Ch(\mathcal{DP}(R))$ and $\Ch(\mathcal{DI}(R))$ coincide with the classes of Ding projective and Ding injective complexes for any ring $R$ (see \cite[Thm. 2.8]{YangDa18} and Gillespie and Iacob's \cite[Thm. 2]{GillespieIacobDingEnvelopes}). For Gorenstein AC-projective and Gorenstein AC-injective complexes, it was proved in \cite[Thm. 3.2 \& 4.13]{BravoGillespieComplexes} that a chain complex $X_\bullet$ is Gorenstein AC-injective (resp., Gorenstein AC-projective) if, and only if, $X_\bullet \in \Ch(\mcGI_{\rm AC}(R))$ (resp., $X_\bullet \in \Ch(\mcGP_{\rm AC}(R))$) and every chain map $A_\bullet \to X_\bullet$ (resp., $X_\bullet \to L_\bullet$) is null homotopic for every absolutely clean complex $A_\bullet$ (resp., for every level complex $L_\bullet$). We are not aware if the null homotopy condition can be removed, as it occurs for Gorenstein and Ding complexes. 
\end{example}

%%%%%%%%%%%%%%%%%%%%%%%%%%%%%%%%%%%%%%%%%%%%%%%%
%%%%%%%%%%%%%%%%%%%%%%%%%%%%%%%%%%%%%%%%%%%%%%%%

\subsection*{Inner balance over chain complexes}

Given two complexes $C_\bullet, D_\bullet \in \Ch (\mcA)$ over an abelian category $\mcA$, let $\Hin(C_\bullet,D_\bullet)$ denote the complex with entries 
\[
\Hin(C_\bullet,D_\bullet)_m = \prod_{i \in \mathbb{Z}} \Hom_{\mcA}(C_i,D_{i+m})
\]
for every $m \in \mathbb{Z}$, and differentials defined by
\[
\partial_m^{\Hin(C_\bullet,D_\bullet)}(f_\bullet) = (\partial_{m+i}^{D_\bullet} f_i -(-1)^{m} f_{i-1} \partial_i^{C_\bullet}) _{i \in \mathbb{Z}}
\]
for every $f_\bullet \in \Hin(C_\bullet,D_\bullet)_m$. Let $\mcX$ and $\omega$ be classes of objects in $\mcA$, where $\mcX$ closed under extensions and direct summands. If $\mcA$ has enough injective objects, $\omega$ is a relative cogenerator in $\mcX$ and $\id_{\mcX}(\omega) = 0$, then by Proposition \ref{rem:complexes}-(2) and \cite[Thm. 2.8]{BMS}, we have that every $C_\bullet \in \Ch(\mcX)^{\wedge}$ has a special $\Ch(\mcX)$-precover $X_{\bullet} \to C_\bullet$ with kernel in $\widetilde{\omega}^\wedge$. Thus using Lemma \ref{ComparisonT}, for each $D_\bullet \in \Ch(\mcA)$ we can define the $i$-th relative cohomology group (with $i \in \mathbb{Z}$) by 
\begin{align}\label{inner_ext_X}
\underline{\Exin}^i_{\mcX}(C_\bullet, D_\bullet) & := \mathrm{H}^i(\Hin(X_{\bullet}, D_\bullet)).
\end{align}
With dual properties and arguments, provided that $\mcA$ has enough projective objects, for each $D_\bullet \in \Ch(\mcY)^{\vee}$ and $C_\bullet \in \Ch(\mcA)$ we can define the $i$-th relative cohomology group by
\begin{align}\label{inner_ext_Y}
\overline{\Exin}^i_{\mcY}(C_\bullet, D_\bullet) & := \mathrm{H}^i(\Hin(C_\bullet, Y_{\bullet})),
\end{align}
where $D_\bullet \to Y_\bullet$ is a special $\Ch(\mcY)$-preenvelope of $D_\bullet$ with cokernel in $\widetilde{\nu}^\vee$. These definitions of $\underline{\Exin}^i_{\mcX}(-,\sim)$ and $\overline{\Exin}^i_{\mcY}(-,\sim)$ were originally introduced by Di, Lu and Zhao in \cite[Def. 3.4]{DiLuZhao20}. 

In this section, we show an analog for Proposition \ref{CriterioBalance} which involves the ``inner'' extensions functors defined in \eqref{inner_ext_X} and \eqref{inner_ext_Y}, that is, both definitions coincide under certain conditions. More precisely, we shall require that $\mcX$ and $\mcY$ are part of a strong finite balanced system. 

The following lemma can be proven as in Liu's \cite[Lem. 11]{Liu18}, or by using \cite[Lem. 2.1]{GillespieFlat}. Although the arguments there are stated in the category of modules, they carry over to abelian categories.

\begin{lemma} \label{HomTrivial}
Let $X_{\bullet}, L_{\bullet} \in \Ch(\mcA)$. If $\Ext^1_{\Ch}(X_{\bullet}, L_{\bullet}[-i]) = 0$ for every $i \in \mathbb{Z}$, then $\Hin(X_{\bullet}, L_{\bullet})$ is an exact complex.
\end{lemma}

The previous lemma, along with Lemma \ref{BalanceComplejo} and the fact that $\widetilde{\nu}^\vee$ is closed under suspensions, implies the following result.

\begin{corollary}\label{CorolarioHomTrivial}
Let $\mcA$ be an abelian category with enough injective objects, and $\mcX$, $\omega$ and $\nu$ be classes of objects in $\mcA$ satisfying the list of conditions in Lemma \ref{BalanceComplejo}. Then, for every $X_{\bullet} \in \Ch (\mcX)$ and $L_{\bullet} \in \widetilde{\nu}^\vee$, the complex $\Hin(X_{\bullet}, L_{\bullet})$ is exact.
\end{corollary}

\begin{theorem}[relative inner extension bifunctors and balance]
Let $\mcA$ be an abelian category with enough projective and injective objects and consider a strong finite balanced system $[(\mcX , \omega); (\nu , \mcY)]$ in $\mcA$, with $\mcX$ and $\mcY$ closed under direct summands, and $\omega$ and $\nu$ closed under finite coproducts. Then, for every $C_{\bullet} \in \Ch (\mcX)^{\wedge}$, $D_{\bullet} \in \Ch(\mcY)^{\vee}$ and $i \in \mathbb{Z}$, there is a natural isomorphism
\[
\underline{\Exin}^{i}_{\mcX}(C_{\bullet}, D_{\bullet}) \cong \overline{\Exin}^{i}_{\mcY}(C_{\bullet}, D_{\bullet}).
\] 
\end{theorem}

The previous result was originally proved in \cite[Thm. 3.5]{DiLuZhao20} under slightly different hypotheses. Namely, the authors assume $\underline{\Ext}^{\geq 1}_{\omega}(\omega^\wedge,\nu) = 0$ and $\overline{\Ext}^{\geq 1}_{\nu}(\omega,\nu^\vee) = 0$, while we assume instead conditions $\mathsf{(fbs5)}$ and $\mathsf{(fbs6)}$ in Definition \ref{finite-balanced-system}, along with $\pd_{\mcY}(\omega) = 0$ and $\id_{\mcX}(\nu) = 0$.

\begin{proof}
First, let us consider for $D_{\bullet} \in \Ch (\mcY)^{\vee}$ a short exact sequence 
\[
D_{\bullet} \rightarrowtail Y_{\bullet} \stackrel{\lambda_\bullet} \twoheadrightarrow L_{\bullet},
\] 
with $Y_{\bullet} \in \Ch(\mcY)$ and $L_{\bullet} \in \widetilde{\nu}^{\vee}$, by \cite[Thm. 2.8]{BMS}. We show that 
\[
\Hin (X_{\bullet} , \lambda_\bullet) \colon \Hin (X_{\bullet}, Y_{\bullet}) \to \Hin(X_{\bullet}, L_{\bullet})
\] 
is an epimorphism for every $X_{\bullet} \in \Ch (\mcX)$. We know that 
\[
\Hin(X_{\bullet}, \lambda_\bullet)_{m} = \prod_{i \in \mathbb{Z}} \Hom_{\mcA} (X_i, \lambda _{i+m}),
\]
for each $m \in \mathbb{Z}$, since the morphism $\Hin(X_{\bullet}, \lambda_\bullet)_{m}$ is defined by
\[
\Hin(X_{\bullet}, \lambda_\bullet)_{m}(f_\bullet) := (\lambda_{i+m} \circ f_i)_{i \in \mathbb{Z}}
\]
for every $f_\bullet = (f_i)_{i \in \mathbb{Z}} \in \Hin (X_{\bullet}, Y_{\bullet})_m$. Thus, it is enough to show that for each $i \in \mathbb{Z}$ the morphism 
\[
\Hom_{\mcA}(X_i, \lambda _{i+m}) \colon \Hom_{\mcA} (X_i, Y_{i+m}) \to \Hom_{\mcA}(X_i, L_{i+m})
\] 
is an epimorphism, as the product of epimorphisms in $\Mod(\mathbb{Z})$ is again epic. 

From $\id_{\Ch(\mcX)}(\widetilde{\omega}) = 0$ and that $\widetilde{\omega}$ is a relative cogenerator in $\Ch(\mcX)$ (Proposition \ref{rem:complexes}-(2)) there is a short exact sequence $X_{\bullet} \stackrel{\alpha_\bullet}\rightarrowtail W_{\bullet} \twoheadrightarrow X''_{\bullet}$ with $X''_{\bullet} \in \Ch (\mcX)$ and $W_{\bullet} \in \widetilde{\omega}$. Then for each $i \in \mathbb{Z}$, we have the following family of exact sequences
\[
X_i \stackrel{\alpha _{i}} \rightarrowtail W_i \twoheadrightarrow X''_i \;\;\mbox{and}\;\; D_{i+m} \rightarrowtail Y_{i+m} \stackrel{\lambda _{i+m}} \twoheadrightarrow L_{i+m}
\]
where $X_i , X''_i \in \mcX$, $Y_{i+m} \in \mcY$, $W_i \in \omega$, $D_{i+m} \in \mcY^{\vee}$ and $L_{i+m} \in \nu^{\vee}$. For the latter relation $L_{i+m} \in \nu^{\vee}$, note that for $L_\bullet$ there is an exact sequence 
\[
L_\bullet \rightarrowtail V^0_\bullet \to V^1_\bullet \to \cdots \to V^{t-1}_\bullet \twoheadrightarrow V^t_\bullet
\] 
for some $t \in \mathbb{N}$, where each $V^i_\bullet$ belongs to $\widetilde{\nu}$. Since $\Ext^1_{\mcA}(\nu,\nu) = 0$, every complex in $\widetilde{\nu}$ is a direct sum of disk complexes centered at objects in $\nu$. Also, since $\nu$ is closed under finite coproducts, we have that $V^i_m \in \nu$ for every $m \in \mathbb{Z}$. Now consider the following commutative diagram 
\[
\begin{tikzpicture}[description/.style={fill=white,inner sep=2pt}] 
\matrix (m) [matrix of math nodes, row sep=3em, column sep=6.5em,text height=1.25ex, text depth=0.25ex] 
{ 
\Hom_{\mcA}(W_i,Y_{i+m}) & \Hom_{\mcA}(X_i,Y_{i+m}) \\
\Hom_{\mcA}(W_i,L_{i+m}) & \Hom_{\mcA}(X_i,L_{i+m}) \\
}; 
\path[->]
(m-1-1) edge node[above] {\footnotesize$\Hom_{\mcA}(\alpha_i,Y_{i+m})$} (m-1-2) edge node[left] {\footnotesize$\Hom_{\mcA}(W_i,\lambda_{i+m})$} (m-2-1)
(m-1-2) edge node[right] {\footnotesize$\Hom_{\mcA}(X_i,\lambda_{i+m})$} (m-2-2)
(m-2-1) edge node[below] {\footnotesize$\Hom_{\mcA}(\alpha_i,L_{i+m})$} (m-2-2)
;
\end{tikzpicture} 
\]
By Remark \ref{rem:sfbs}, we have that $\id_{\omega}(\mcY^\vee) = \id_{\mcX}(\nu^\vee) = 0$, and so $\Hom_{\mcA}(W_i, \lambda_{i+m})$ and $\Hom_{\mcA}(\alpha_i, L_{i+m})$ are epimorphisms. Hence, so is $\Hom_{\mcA}(X_i, \lambda_{i+m})$.  

So far we have proved that for every $D_{\bullet} \in \Ch(\mcY)^{\vee}$ there is a short exact sequence $D_{\bullet} \rightarrowtail Y_{\bullet} \twoheadrightarrow  L_{\bullet}$, such that $\Hin(X_{\bullet}, D_{\bullet}) \rightarrowtail \Hin (X_{\bullet}, Y_{\bullet}) \twoheadrightarrow \Hin(X_{\bullet}, L_{\bullet})$ is a short exact sequence of complexes for every $X_{\bullet} \in \Ch (\mcX)$. Applying Corollary \ref{CorolarioHomTrivial} to the complex $\Hin(X_{\bullet}, L_{\bullet})$, we conclude that 
\[
\mathrm{H}^{i} (\Hin (X_{\bullet}, D_{\bullet})) \cong \mathrm{H}^{i}(\Hin (X_{\bullet}, Y_{\bullet})), \text{ for every } i \in \mathbb{Z}.
\]
Now let $C_{\bullet} \in \Ch(\mcX)^\wedge$. Using a dual argument, there is a short exact sequence $K_{\bullet} \rightarrowtail X'_{\bullet} \twoheadrightarrow C_{\bullet}$, with $X'_{\bullet} \in \Ch(\mcX)$ and $K_{\bullet} \in \widetilde{\omega}^{\wedge}$, such that the sequence of abelian groups $\Hin(C_{\bullet}, Y_{\bullet}) \rightarrowtail \Hin(X'_{\bullet}, Y_{\bullet}) \twoheadrightarrow \Hin(K_{\bullet},Y_{\bullet})$ is exact, where the last term is an exact complex by the dual of Corollary \ref{CorolarioHomTrivial}. Then, we obtain
\[
\mathrm{H}^{i}(\Hin(C_{\bullet}, Y_{\bullet})) \cong \mathrm{H}^{i}(\Hin(X'_{\bullet}, Y_{\bullet})).
\]
Therefore, $\mathrm{H}^{i}(\Hin(X'_{\bullet}, D_{\bullet})) \cong \mathrm{H}^{i}(\Hin(X'_{\bullet}, Y_{\bullet})) \cong \mathrm{H}^i(\Hin(C_\bullet,Y_\bullet))$, that is, 
\[
\underline{\Exin}^{i}_{\mcX}(C_{\bullet}, D_{\bullet}) \cong \overline{\Exin}^i_{\mcY}(C_\bullet,D_\bullet).
\]
\end{proof}

%%%%%%%%%%%%%%%%%%%%%%%%%%%%%%%%%%%%%
%%%%%%%%%%%%%%%%%%%%%%%%%%%%%%%%%%%%%

\subsection*{Virtually Gorenstein balanced pairs}

From Zareh-Khoshchehreh, Asgharzadeh and Divaani-Aazar's \cite[Def. 3.9]{Fatem14}, a commutative noetherian ring $R$ of finite Krull dimension such that $\mcGP(R)^{\perp _1} = {}^{\perp_1} \mcGI(R)$ is called \emph{virtually Gorenstein}. In this work, the authors proved that over a commutative noetherian ring $R$ of finite Krull dimension, the pair $(\mcGP(R), \mcGI(R))$ is balanced with respect to the pair $(\Mod(R), \Mod(R))$ if and only if $R$ is virtually Gorenstein (see \cite[Thm. 3.10]{Fatem14}). In what follows, we show that the conditions of being noetherian and having finite Krull dimension can be replaced by other general conditions in the setting of abelian categories. To that end, let us present some  terminology. Given a class $\mcX$ of objects in an abelian category $\mcA$, the class $\LPRes^{0}(\mcX) \subseteq \LPRes(\mcX)$ is formed by those objects $M \in \LPRes(\mcX)$ for which there exists an exact left $\mcX$-resolution $\varepsilon \colon \mcX_{\bullet}(M) \twoheadrightarrow M$ such that ${\rm Ker}(\varepsilon) \in \mcX^{\perp}$ and $Z_i(\mcX_{\bullet}(M)) \in \mcX^{\perp}$ for every $i \in \mathbb{N}^\ast$. Dually, we have the class $\RPRes^{0}(\mcX) \subseteq \RPRes(\mcX)$.

\begin{theorem}\label{virtual_balance}
Let $\mcA$ be an abelian category with enough projective and injective objects, and $\mcX$ and $\mcY$ be classes of objects in $\mcA$ such that $\mcX^{\perp _1} = {}^{\perp _1}\mcY$. Then, $(\mcX, \mcY)$ is an admissible balanced pair with respect to $(\LPRes^{0} (\mcX), \RPRes^{0} (\mcY))$.
\end{theorem}

\begin{proof}
Let us take $M \in \LPRes(\mcX)^{0}$, and let $\varepsilon \colon \mcX_{\bullet}(M) \twoheadrightarrow M$ be an exact left $\mcX$-resolution with ${\rm Ker}(\varepsilon) \in \mcX^{\perp}$ and $Z_i(\mcX_{\bullet}(M)) \in \mcX^{\perp}$ for every $i \in \mathbb{N}^\ast$. We show that $\varepsilon \colon\mcX_{\bullet} (M) \twoheadrightarrow M$ is $\Hom _{\mcA}(-,\mcY)$-acyclic. For, it suffices to check that the short exact sequence ${\rm Ker}(\varepsilon) \rightarrowtail X_0  \twoheadrightarrow M$ is $\Hom_{\mcA}(-,\mcY)$-acyclic. Since $\mcA$ has enough injective objects, there exists a short exact sequence ${\rm Ker}(\varepsilon) \rightarrowtail E \twoheadrightarrow K$ with $E \in \mcI(\mcA)$. Moreover, since ${\rm Ker}(\varepsilon), E \in \mcX^{\perp}$, we have that $K \in \mcX^{\perp} \subseteq {^{\perp _1} \mcY}$. Then, the sequence ${\rm Ker}(\varepsilon) \rightarrowtail E \twoheadrightarrow K$ is $\Hom_{\mcA}(-,\mcY)$-acyclic. On the other hand, using the fact that $E$ is injective along with the universal property of cokernels, we can obtain the following commutative diagram with exact rows:
\[
\begin{tikzpicture}[description/.style={fill=white,inner sep=2pt}] 
\matrix (m) [matrix of math nodes, row sep=2em, column sep=1.5em,text height=1.25ex, text depth=0.25ex] 
{ 
{\rm Ker}(\varepsilon) & X_0 & M \\
{\rm Ker}(\varepsilon) & E & K \\
}; 
\path[->]
(m-1-2) edge (m-2-2)
(m-1-3) edge (m-2-3)
;
\path[>->]
(m-1-1) edge (m-1-2)
(m-2-1) edge (m-2-2)
;
\path[->>]
(m-1-2) edge (m-1-3)
(m-2-2) edge (m-2-3)
;
\path[-,font=\scriptsize]
(m-1-1) edge [double, thick, double distance=2pt] (m-2-1)
;
\end{tikzpicture} 
\]
Applying $\Hom_{\mcA}(-,Y)$ with $Y \in \mcY$, we obtain the following commutative diagram
\[
\begin{tikzpicture}[description/.style={fill=white,inner sep=2pt}] 
\matrix (m) [matrix of math nodes, row sep=2em, column sep=1.5em,text height=1.25ex, text depth=0.25ex] 
{ 
\Hom_{\mcA}(K,Y) & \Hom_{\mcA}(E,Y) & \Hom_{\mcA}({\rm Ker}(\varepsilon),Y) \\
\Hom_{\mcA}(M,Y) & \Hom_{\mcA}(X_0,Y) & \Hom_{\mcA}({\rm Ker}(\varepsilon),Y) \\
}; 
\path[->]
(m-1-1) edge (m-2-1)
(m-1-2) edge (m-2-2)
(m-2-2) edge (m-2-3)
;
\path[>->]
(m-1-1) edge (m-1-2)
(m-2-1) edge (m-2-2)
;
\path[->>]
(m-1-2) edge (m-1-3)
;
\path[-,font=\scriptsize]
(m-1-3) edge [double, thick, double distance=2pt] (m-2-3)
;
\end{tikzpicture} 
\]
where the upper row is exact. It follows that the sequence ${\rm Ker}(\varepsilon) \rightarrowtail X_0 \twoheadrightarrow M$ is $\Hom_{\mcA}(-,\mcY)$-acyclic. Condition $\mathsf{(bp3)}$ in Definition \ref{def:balanced_pair} follows dually. 
\end{proof}

Under certain conditions, the relative balanced pair in Theorem \ref{virtual_balance} can be obtained from a  balanced system. The proof of the following result is straightforward.

\begin{proposition}
Let $\mcX$ and $\mcY$ be classes of objects in an abelian category $\mcA$ with enough projective and injective objects, satisfying 
\[
\mcX^{\perp_1} = {}^{\perp_1}\mcY.
\] 
If $\mcX \cap \mcX^{\perp_1}$ is a relative cogenerator in $\mcX$ and ${}^{\perp_1}\mcY \cap \mcY$ is a relative generator in $\mcY$, then 
\[
[(\mcX,\mcX \cap \mcX^{\perp_1});({}^{\perp_1}\mcY \cap \mcY,\mcY)]
\] 
is a balanced system with respect to the couple 
\[
[(\mcX \cap \mcX^{\perp_1},\LPRes^0(\mcX));(\RPRes^0(\mcY),{}^{\perp_1}\mcY \cap \mcY)].
\]
\end{proposition}

%%%%%%%%%%%%%%%%%%%%%%%%%%%%%%%%%%%%%%%%%%%%%%%%
%%%%%%%%%%%%%%%%%%%%%%%%%%%%%%%%%%%%%%%%%%%%%%%%
%%%%%%%%%%%%%%%%%%%%%%%%%%%%%%%%%%%%%%%%%%%%%%%%
%%%%%%%%%%%%%%%%%%%%%%%%%%%%%%%%%%%%%%%%%%%%%%%%

\subsection*{Balance over Cohen-Macaulay rings}

Throughout, consider a commutative Cohen-Macaulay ring $R$ with unit. For a fixed $R$-module $C$, the classes of $C$-projective and $C$-injective $R$-modules were defined by Holm and J{\o}rgensen in \cite{HolmJorgensen06} as follows:
\begin{align*}
\mcP_{C}(R) & := \{ C \otimes_{R} P \ : \ P \in \mathcal{P}(R) \} & & \text{and} & \mcI_{C}(R) & := \{ \Hom_{R} (C, I) \ : \ I \in \mathcal{I}(R) \}.
\end{align*}
Based on these classes, they also introduced the classes of $C$-Gorenstein projective and $C$-Gorenstein injective $R$-modules, which will be denoted by $\mathcal{GP}_C(R)$ and $\mathcal{GI}_C(R)$, respectively (see \cite[Def. 2.7]{HolmJorgensen06}). 

Associated to a semidualizing $R$-module $C$, we have the Auslander and Bass classes (see for instance Takahashi and White's \cite[\S 1.8]{TWhite10}). The \textit{Bass class respect to $C$}, denoted $\mcB_{C}(R)$, consists of all $R$-modules $M$ satisfying
\[
\Ext^{\geq 1}_R(C,M) = 0 = \mathrm{Tor}^R_{\geq 1}(C, \Hom_{R}(C, M)),
\]
and such that the natural evaluation map $\nu_{M} \colon C \otimes_R \Hom_{R}(C, M) \to M$ is an isomorphism. Dually, the \textit{Auslander class with respect to $C$}, denoted $\mcA_{C}(R)$, consists of all $R$-modules $M$ satisfying
\[
\mathrm{Tor}_{\geq 1}^R(C,M) = 0 = \Ext^{\geq 1}_R(C, C \otimes_R M),
\]
and such that the natural map $\mu_M \colon M \to \Hom_{R}(C, C \otimes_R M)$ is an isomorphism. With these notions in mind, we give the following example of strong finite balanced system.

\begin{proposition}
Let $R$ be a commutative Cohen-Macaulay ring with a dualizing $R$-module $D$ and a semidualizing $R$-module $C$, and let $C^{\dag} := \Hom_{R}(C,D)$. Then, 
\[
[(\mathcal{GP}_{C}(R),\mathcal{P}_{C}(R));(\mathcal{I}_{C^{\dag}}(R), \mathcal{GI}_{C^{\dag}}(R))]
\] 
is a strong finite balanced system. In particular, $(\mathcal{GP}_{C}(R), \mathcal{GI}_{C^{\dag}}(R))$ is a balanced pair with respect to $(\mathcal{GP}_{C}(R)^\wedge, \mathcal{GI}_{C^{\dag}}(R)^{\vee})$.
\end{proposition}

\begin{proof}
Let us check first conditions from $\mathsf{(fbs1)}$ to $\mathsf{(fbs4)}$ in Definition \ref{finite-balanced-system}. The closure under extensions of $\mathcal{GP}_C(R)$ was proved by White in \cite[Thm. 2.8]{White10}, who also showed, along with Sather-Wagstaff and Sharif in \cite[Fact 2.4]{SeanWhite08}, that $\mathcal{P}_C(R)$ is a relative cogenerator in $\mathcal{GP}_{C}(R)$. Moreover, $\id_{\mathcal{GP}_{C}(R)}(\mathcal{P}_C(R)) = 0$ by \cite[Prop. 2.2]{White10}. Thus, conditions $\mathsf{(fbs1)}$ and $\mathsf{(fbs3)}$ hold. Dually, one also has $\mathsf{(fbs2)}$ and $\mathsf{(fbs4)}$, since $C^{\dag}$ is semidualizing by \cite[Not. 3.1]{SeanWhite10}.

In order to conclude the proof, let us now check conditions (1) to (5) in Proposition \ref{prop:characterization_fbs}. We first verify that $\Ext^{\geq 1}_R(\mathcal{P}_C(R)^\wedge,\mathcal{I}_{C^\dag}(R)) = 0$. So let $M \in \mathcal{P}_C(R)^\wedge$ and $\Hom_{R}(C^\dag,I) \in \mathcal{I}_{C^\dag}(R)$ (that is, $I \in \mathcal{I}(R)$). Note that $D \in \mathcal{B}_C(R)$ since $D$ has finite injective dimension (see \cite[Fact 2.8]{SeanWhite08}). Then we can apply \cite[Lem. 6.14]{SeanWhite08} that asserts that $\mathcal{P}_{C}(R)^\wedge \subseteq \mathcal{A}_{C^\dag}(R)$. On the other hand, by Holm and White's \cite[Coroll. 6.1]{HolmWhite07} we have that $\mathcal{I}_{C^\dag}(R) \subseteq \mathcal{I}_{C^\dag}(R)^\vee \subseteq \mathcal{A}_{C^\dag}(R)$. Then,
\[
\Ext^i_{R}(M,\Hom_R(C^\dag,I)) \cong \overline{\Ext}^i_{\mathcal{I}_{C^\dag}(R)}(M,\Hom_R(C^\dag,I)) = 0,
\]
for every $i \in \mathbb{N}^\ast$, where the isomorphism results from \cite[Coroll. 4.2]{TWhite10}. Dually, we have $\Ext^{\geq 1}_R(\mathcal{P}_C(R),\mathcal{I}_{C^{\dag}}(R)^\vee) = 0$. 

Now let us show that $\Ext^{\geq 1}_R(\mathcal{GP}_{C}(R),\mathcal{I}_{C^{\dag}}(R)) = 0$. By \cite[Thm. 4.6]{HolmJorgensen06}, we have the containment $\mathcal{GP}_{C}(R) \subseteq \mathcal{A}_{C^{\dag}}(R)$. We also know that $\mathcal{I}_{C^\dag}(R) \subseteq \mathcal{A}_{C^\dag}(R)$. Then, the result follows again by \cite[Coroll. 4.2]{TWhite10}, while $\Ext^{\geq 1}_R(\mathcal{P}_C(R),\mathcal{GI}_{C^{\dag}}(R)) = 0$ is dual. The latter implies that $\pd_{\mathcal{GI}_{C^\dag}(R)}(M) = 0$ and $\id_{\mathcal{GP}_{C}(R)}(N) = 0$ for every $M \in \mathcal{P}_C(R)$ and $N \in \mathcal{I}_{C^\dag}(R)$. Using a dimension shifting argument, we can deduce that $\pd_{\mathcal{GI}_{C^\dag}(R)}(M) < \infty$ and $\id_{\mathcal{GP}_{C}(R)}(N) < \infty$ for every $M \in \mathcal{P}_C(R)^\wedge$ and $N \in \mathcal{I}_{C^\dag}(R)^\vee$. Moreover, in particular we have $\id_{\mathcal{P}_C(R)}(\mathcal{I}_{C^{\dag}}(R)) = 0$, since $\mathcal{P}_C(R) \subseteq \mathcal{GP}_C(R)$. 
\end{proof}

The previous proposition, along with Proposition \ref{CriterioBalance} implies that 
\begin{align}
\underline{\Ext}^i_{\mathcal{GP}_C(R)}(M,N) & \cong \overline{\Ext}^i_{\mathcal{GI}_{C^\dag}(R)}(M,N).
\end{align}
for every $M \in \mathcal{GP}_{C}(R)^\wedge$ and $N \in \mathcal{GI}_{C^\dag}(R)^\vee$, getting thus another proof of \cite[Thm. 5.7]{SeanWhite10} (also by Sather-Wagstaff, Sharif and White).

%%%%%%%%%%%%%%%%%%%%%%%%%%%%%%%%%%%%%
%%%%%%%%%%%%%%%%%%%%%%%%%%%%%%%%%%%%%
%%%%%%%%%%%%%%%%%%%%%%%%%%%%%%%%%%%%%
%%%%%%%%%%%%%%%%%%%%%%%%%%%%%%%%%%%%%

%\section*{\textbf{Acknowledgements}}

%%%%%%%%%%%%%%%%%%%%%%%%%%%%%%%%%%%%%
%%%%%%%%%%%%%%%%%%%%%%%%%%%%%%%%%%%%%
%%%%%%%%%%%%%%%%%%%%%%%%%%%%%%%%%%%%%
%%%%%%%%%%%%%%%%%%%%%%%%%%%%%%%%%%%%%

\section*{\textbf{Funding}}

The authors thank Project PAPIIT-Universidad Nacional Aut\'onoma de M\'exico IN100124. The first author was partially supported by a postdoctoral fellowship from Programa de Desarrollo de las Ciencias B\'asicas (PEDECIBA). The third author was partially supported by the following grants and institutions: Fondo Vaz Ferreira \# II/FVF/2019/135 (funds are given by the Direcci\'on Nacional de Innovaci\'on, Ciencia y Tecnolog\'ia - Ministerio de Educaci\'on y Cultura, and administered through Fundaci\'on Julio Ricaldoni), Agencia Nacional de Investigaci\'on e Innovaci\'on (ANII), and PEDECIBA.

%%%%%%%%%%%%%%%%%%%%%%%%%%%%%%%%%%%%%
%%%%%%%%%%%%%%%%%%%%%%%%%%%%%%%%%%%%%
%%%%%%%%%%%%%%%%%%%%%%%%%%%%%%%%%%%%%
%%%%%%%%%%%%%%%%%%%%%%%%%%%%%%%%%%%%%

\bibliographystyle{plain}
\bibliography{biblio_balance_system}

\end{document}